\documentclass[leqno,11pt]{amsart}
\usepackage{amsmath}
\usepackage{graphicx, color}
\usepackage{amscd}
\usepackage{amsfonts}
\usepackage{amssymb}
\usepackage{mathrsfs}
\usepackage{mathtools}
\usepackage{ulem}
\usepackage{needspace}

\newcommand{\bel}{\begin{equation} \label}

\newcommand{\cbl}{\color{black}}

\newtheorem{thm}{Theorem}[section]

\newtheorem{lem}[thm]{Lemma}
\newtheorem{prop}[thm]{Proposition}

\newtheorem{rem}{Remark}[section]

\numberwithin{equation}{section}
\newcommand{\ep}{\varepsilon}
\newcommand\be{\begin{equation}}
\newcommand\ee{\end{equation}}
\newcommand\R{\mathbb R}

\newcommand\ds{\displaystyle}

\allowdisplaybreaks
\def\eps{\varepsilon}

\newcommand\rn{\R^n}

\newcommand\ld{\mathcal{L}}
\newcommand\ldo{\mathcal{L}_0}

\newcommand\ldu{\ld u}

\title[Differential Harnack inequalities and quantitative maximum principles]
{Differential Harnack inequalities and\\   maximum principles of Morel-Oswald type\\   for  elliptic PDE in divergence form}

\author[Sirakov]{Boyan Sirakov}
\address{PUC-Rio, Departamento de Matematica \\
Rua Marqu\^es de S\~ao Vicente 225 \\
G\'avea, Rio de Janeiro - CEP 22451-900, Brazil}
\email{bsirakov@puc-rio.br}

\author[Souplet]{Philippe Souplet}%
\address{Universit\'e Sorbonne Paris Nord,
CNRS UMR 7539, Laboratoire Analyse, G\'{e}om\'{e}trie et Applications,
93430 Villetaneuse, France}
\email{souplet@math.univ-paris13.fr}

\begin{document}

\begin{abstract}   This paper presents 
 two types of novel, qualitative and quantitative, estimates for positive solutions of general uniformly elliptic PDEs in divergence form.
First, we prove a Morel-Oswald type of extension of the  Hopf-Oleinik lemma, in which in addition we specify the
 sharp dependence of the constant in the data of the operator and the size of the domain. Second, we establish a new differential Harnack  (logarithmic gradient) estimate for non-homogeneous equations, as well as an optimal global differential Harnack estimate  for homogeneous equations.
\end{abstract}

\maketitle

\section{Introduction}\label{sec-intro}
This paper is a contribution to the study of  qualitative and quantitative estimates for positive (super-)solutions of
uniformly elliptic  equations in general divergence form
\begin{equation}\label{defdiv}
\ldu:=\mathrm{div}(A(x)\nabla u +  b_1(x)u) + b_2(x)\cdot \nabla u +c(x) u = f(x),
\end{equation}
in a bounded domain $\Omega\subset\rn$  with $C^{1,\bar\alpha}$-boundary
for some $\bar\alpha\in (0,1]$.
Since our results involve gradient estimates we make the following classical assumptions on the coefficients (possible generalizations are discussed below):
\begin{eqnarray}
&\hskip 1cm\hbox{$A(x)$ is a  matrix such that $\Lambda I \ge A\ge\lambda I$ in $\Omega$, for some $\Lambda\ge \lambda>0$,} \label{hyp1} \\
\noalign{\vskip 1mm}
&\hskip 1cm\hbox{$A, b_1\in C^{\alpha}(\Omega)$} , \quad\hbox{$ b_2, c,f \in L^q(\Omega)$}, \quad\hbox{for some $q\in(n,\infty]$,  
$\alpha\in(0,1)$,} \label{hyp2}
\end{eqnarray}
which guarantee that the weak ($u\in H^1(\Omega)$) solutions of \eqref{defdiv} are continuously differentiable.

We  study {\it pointwise gradient bounds for positive solutions}. Specifically, we are interested in explicit 
upper bounds for the gradient $|\nabla u(x)|$ at each point $x\in\overline{\Omega}$, in terms of the value of the solution $u(x)$ at the same point. Such estimates are usually referred to as differential Harnack inequalities (DHI), and provide much more precise information than the usual $C^1$-estimates in which $\sup|\nabla u|$ is bounded by $\sup u$. 
While many DHI are available for operators with regular coefficients, much less is known for the case when the coefficients are not smooth, as we explain below. 
Second, we obtain positive lower bounds for the gradient  on $\partial\Omega$ of a positive solution of the Dirichlet problem, or more generally for the quantity $u(x)/\mathrm{dist}(x,\partial \Omega)$ in $\overline{\Omega}$ for any positive supersolution of  \eqref{defdiv}. Such results fall under the general category of quantitative strong maximum principles; here most of our effort (though not all) will focus on the so-called Morel-Oswald estimate.

In the rest of the introduction we give an overview of our contributions to these two topics, starting with the latter, since it will be used for the former.

\subsection{Optimized Morel-Oswald estimate (quantitative Hopf-Oleinik lemma)} We start by  recalling a fundamental property of superharmonic functions --- that positivity entails a quantitative version of itself.  Specifically, if $u\ge0$ is $\ld$-superharmonic in a bounded $C^{1,\bar\alpha}$-domain $\Omega\subset \rn$, then for some $c_0=c_0(u,\ld,\Omega)>0$
\begin{equation}\label{hopf1}
u(x)\ge c_0\, d(x),\quad x\in\Omega 
\qquad (\mbox{unless }\; u\equiv0).
\end{equation}
In other words, if $u$ attains a zero minimum inside $\Omega$ then $u\equiv0$ and if $u$ attains a zero minimum at a boundary point, then the normal derivative of $u$ at this point does not vanish. The latter is the famous Zaremba-Hopf-Oleinik lemma, also called boundary point lemma or boundary point principle (BPP). The study of this ``bedrock" (to quote page 1 of \cite{PS}) result in the theory of elliptic PDE spans more than a century -- we refer to the  survey \cite{AN}, as well as to \cite{FG}, \cite{G1}, \cite{N1},  \cite{AN2}, \cite{Sa},   \cite{GSS}, for (both positive and negative) results on the validity of \eqref{hopf1} in the  case of  divergence-form equations considered here.

Note \eqref{hopf1} is only a qualitative estimate, and specifying the constant $c_0(u,\mathcal{L},\Omega)$ in \eqref{hopf1} is a fundamental question.
The best-known way of  quantifying this constant is expressing its dependence in $u$ through the positivity of $u$ itself (an integral norm of $u$). This  leads to estimates of weak Harnack type or ``growth lemmas", which are at  the heart of the regularity theory of elliptic PDE. In particular,  the classical De Giorgi-Moser weak Harnack inequality (\cite[Theorem 8.18]{GT}) is an interior version of \eqref{hopf1} and states that $u\ge c_0(\mathcal{L}, K,\Omega)\|u\|_{L^s(K)}$ in $K$, for a compact $K\subset\Omega$ and $s<n/(n-2)$.  In  \cite{SS2} we obtained a further specification of $c_0$ with respect to $\mathcal{L}$ and $\Omega=B_R$ in that interior inequality, and used it to obtain novel results on the V\'azquez strong maximum principle and the Landis conjecture.
Furthermore,  global (that is,  up to the boundary of $\Omega$) variants of the weak Harnack inequality which extend the BPP were recently obtained in the form $u/d\ge c_0(\mathcal{L}, \Omega)\|u/d\|_{L^s(\Omega)}$ in $\Omega$ for some $s>0$, where $d=d(x)=\mathrm{dist}(x,\partial\Omega)$ (\cite{Sir3}, \cite{GSS}). We note that below
we will present (Theorem~\ref{BHIoptim}) and use a version of the result in \cite{GSS} with a constant $c_0(\mathcal{L}, \Omega)$ whose dependence with respect to the norms of the coefficients of $\mathcal{L}$ and the size of $\Omega$ is made explicit. \smallskip

Another important and useful way of quantifying \eqref{hopf1} with respect to $u$ is expressing  $c_0$ through an integral norm of $-\ldu$ instead of $u$. This type of inequality for $\ld = \Delta$  gained a lot of popularity and was used  in various contexts after the work by Brezis and Cabr\'e~\cite{BrC} (as noted in \cite{BrC}, the estimate is due to an unpublished work by Morel and Oswald; see also \cite[Proposition A.4.2]{D1}). Specifically, the result in \cite{BrC} states that if $u\ge0$ in a smooth domain~$\Omega$, $-\Delta u \ge f \ge0$ in $\Omega$, then
\be\label{HopfLapl}
\frac{u(x)}{d(x)} \ge C(n, \Omega){\|f\|_{L^1_d(\Omega)}}\,, \quad x\in \Omega, \qquad \mbox{where }\; \|f\|_{L^1_d(\Omega)}:=\int_\Omega |f(y)|d(y)dy.
\ee
We refer to \cite{Zh}, \cite{Da}, \cite{EDR}, \cite{D1}, \cite{OP}, \cite{KK},   \cite{Sir1} for variants and extensions of this inequality to various operators and contexts. 
A lot more information on the Morel-Oswald inequality and its various applications can be found in  these works and the references therein. We note that the $L^1_d$-norm in \eqref{HopfLapl} cannot be improved (see Remark \ref{Hopf-optL1d}).

Our first main result (Theorems \ref{OptimizedHopf} and \ref{thm4gen}(i) below) establishes the Morel-Oswald inequality for the full general operator in \eqref{defdiv} including both divergence and drift first order terms, possibly with unbounded drifts and potentials. The most novel and original feature of the method we introduce below is that it yields a precise quantitative bound. Specifically, a question which has not been considered before (in spite of its obvious interest), even in simple cases, is a quantification of the constant in the Morel-Oswald inequality in terms of the size of the domain and the norms of the coefficients of the elliptic operator. We will provide an explicit lower bound for this constant which we will see to be  sharp with respect to these quantities.

Our proof of the Morel-Oswald inequality is different from those available to date. It uses a splitting of the domain into a controlled number of small balls with controlled radius such that a suitable rescaling of the operator in each ball satisfies the maximum principle in the rescaled ball, then using  ideas from \cite{BrC}
(where the case $\mathcal{L}=\Delta$ was treated),
as well as additional arguments based on  the global Harnack inequality mentioned above (which ensure the uniformity of the estimates with respect
to the norm of the nonprincipal coefficients and the domain) and a reduction to operators in simpler form, in order to obtain a variant of the theorem in each ball, and finally observing that the $L^1_d$-norm of $f$ in some  ball is comparable (with a constant depending on the right quantities) to its norm in $B_R$.

 For more comments on  possible generalizations of the hypotheses on $\ld$,
 see Remark \ref{remcoefopt}.

\medskip

In order to give a precise statement, we need to introduce some notations. Since these can look somewhat daunting at first reading for  general $\ld$ and $\Omega$, {\it in order to accommodate the reader, here} we will  give a theorem for the particular and frequently encountered operator
\begin{equation}\label{partoper}
 \ldo u:= \mathrm{div}(A(x)\nabla u) + b(x)\cdot\nabla u + c(x) u,
 \end{equation}
 with bounded  coefficients in a ball $B_R$, and $A\in C^\alpha(B_R)$, $f\in L^q(B_R)$, $q>n$. We set
 \begin{equation}\label{partoper_defM}
M_0 = [A]_{\alpha, B_R}^{1/\alpha}+ \|b\|_{L^\infty(B_R)} + \|c\|_{L^\infty(B_R)}^{1/2},
 \end{equation}
where  $[A]_{\alpha,B_R}$ is the standard H\"older bracket in $B_R$.
 We have the following estimate. 

\begin{thm} \label{OptimizedHopf}
If $u\in H^1(B_R)$ is a weak supersolution
of $-\mathcal{L}_0u\ge f$ and $f,u\ge 0$ in $B_R$,
then for some $c_0,C_0>0$ depending only on $n$, $\lambda$, $\Lambda$, $\alpha$, we have
\be\label{conclHopf}
\frac{u(x)}{d(x)} \ge  c_0{e^{-C_0M_0R}\,R^{-n}}\int_{B_R}fd ,\quad x\in B_R.
\ee
\end{thm}

\begin{rem} Extending the quantities $M_0$ and $R$ to the case of a general $\ld$ as in \eqref{defdiv}-\eqref{hyp2} and a general $C^{1, \alpha}$-domain $\Omega$ requires some  work which (in spite of being important for the general theory) we have postponed to the next section, for the sake of simplicity and shortness of this introduction. The theorems we state in the introduction are both new and sufficient to grasp the essence of the more general results in Section \ref{sec-main} below.

We note notwithstanding that obtaining a precise constant in the Morel-Oswald inequality (and the following differential Harnack inequality) in the presence of unbounded coefficients requires additional new ideas; in particular, use of locally uniform Lebesgue spaces with carefully adjusted radius -- see Section \ref{sec-main}. \end{rem}

\begin{rem} We will not need to assume the solvability of the Dirichlet problem for $\ld$ (or the maximum principle), and there is no assumption for $u$ on the boundary.\end{rem}

The following proposition shows that the constant in \eqref{conclHopf} cannot be improved.

\begin{prop} \label{Prop-Optim-Hopf0}
For all $R, M_0>0$, there exist
$c\in C^\infty(\overline B_R)$  with $c\ge 0$ (resp.,~$b\in C^\infty(\overline B_R)$)
and $f\in C^\infty(\overline B_R)$ with $f>0$,
such that the positive classical solution  of
\[\label{optimHopf0}
- \ld_0 u
=f\;\mbox{ in } B_R,\qquad
u=0 \;\mbox{ on } \partial B_R,
\]
with  $\ld_0=\Delta-c$ (resp.,~$\ld_0=\Delta-b\cdot\nabla$)
satisfies, for  $M_0= \|c\|_{L^\infty(B_R)}^{1/2}$ (resp. $M_0= \|b\|_{L^\infty(B_R)}$) and some $C_0=C_0(n)>0$
\be \label{optimHopf2}
\ds\inf_{B_R} \frac{u}{d} \le \sup_{\partial B_R} \left|\frac{\partial u}{\partial \nu}\right| \le C_0\,e^{-C_0M_0R}\, R^{-n}\int_{B_R}fd.
\ee
\end{prop}

Next, we present our results on differential Harnack inequalities, in which we will also use the  above sharp Morel-Oswald estimate.

\subsection{Optimized differential Harnack inequalities (logarithmic gradient estimates)}
Here we study bounds for $|\nabla \log(u)|$, for a positive solution of \eqref{defdiv}. Such estimates are often referred to as differential Harnack inequalities (DHI), since in many cases they can be integrated to obtain the usual Harnack inequality. Fundamental DHI for
 harmonic functions and for the heat equation  (more generally on manifolds) are due to  Cheng-Yau \cite{CY}, and to
Li-Yau \cite{LY} and Hamilton \cite{H}, respectively.
The Li-Yau-Hamilton estimates have generated huge number of extensions and have been successfully adapted
and generalized to various nonlinear geometric heat flows (see for instance the book \cite{RM} and the survey \cite{Ni}).
For the elliptic case, there is a large literature on extensions of these estimates and the related Harnack inequalities to semi-linear equations of Lane-Emden or logarithmic Lane-Emden type (see for instance \cite{LiJ}, \cite{MZ}, \cite{Lu}, \cite{GG} and the references to and from these works). Another classical line of pointwise gradient bounds (Modica-type estimates) for quasilinear  equations is related to use of so-called $P$-functions, see \cite{Sp}, \cite{Mod}, \cite{CGS}, \cite{FV}, \cite{RSW}, among others.

An essential feature in  Cheng-Yau and Li-Yau-Hamilton
type estimates, which leads to their applications, is that the constant in the inequality is made explicit with  respect to the size of the domain and the ambient space. Note that a qualitative version only (i.e.~without specifying the constant) of the DHI for linear equations is an immediate consequence of standard regularity estimates. For instance, if  
$\ld u=0$
and $u>0$ in $B_{3R}$,  combining the usual interior gradient estimate and the Harnack inequality gives 
\begin{equation} \label{qualdhi}|\nabla u(x)|\le\sup_{B_R}|\nabla u|\le C_1\sup_{B_{2R}} u \le C_1C_2\inf_{B_{2R}} u \le C_1C_2 u(x),\quad x\in B_R.
\end{equation}
 However, specifying the optimal $C_1$ and $C_2$ in this inequality shows (as we will recall below) that $C_2$ is in general exponential in $R$ and the coefficients of $\ld$. On the other hand,   Cheng-Yau and Li-Yau-Hamilton estimates (when applicable to $\ld u=0$) show that the  dependence of $C$ in $|\nabla u(x)|\le C u(x)$ with respect to $R$ is at most linear. In other words, some strong cancellation occurs in the fraction $|\nabla u(x)|/u(x)$, which is lost in an argument such as \eqref{qualdhi}, so uncovering it requires additional work and makes these DHI highly nontrivial.

Pointwise gradient estimates are most frequently  proved by variations of the Bernstein method, which involves differentiating the equation and requires smoothness of its coefficients. For that reason the above quoted works require  such smoothness assumptions. On the other hand, DHI for less regular equations were tackled only recently, stemming from the important work by Kenig, Silvestre and Wang  \cite{KSW}. They were motivated by the Landis conjecture which involves only a bounded potential,  thus considered $-\Delta u + c(x)u = 0$ for $c(x)$ bounded measurable  in a two-dimensional ball $B_2\subset \mathbb{R}^2$, and proved the interior estimate $\|\nabla \log(u)\|_{L^\infty(B_1)}\le C_0 \|c\|_{L^\infty(B_2)}^{1/2}$ ($=C_0M_0$ in the above notation).
Extensions of this estimate to more general $2$-dimensional operators as in \eqref{defdiv} are given in \cite{DW}. The same  interior estimate for the $n$-dimensional equation $\ldu=0$ is obtained in \cite{LeB}, for an operator which is also in non-divergence form, and under the restrictions $A=Id$, $b_1=0$, $b_2, c\in L^\infty$. 

We next present our main results on DHI, whose interest is fourfold. First, we remove the restrictions on the dimension and/or the coefficients from the previous works, and prove the estimate for the general operator in \eqref{defdiv}, requiring from the coefficients only the regularity necessary for the gradient to be defined pointwise. Second, we provide a global, up-to-the-boundary differential Harnack bound. Third,  the constant in the estimate  exhibits the expected optimal behaviour -- linear in $M_0$ (actually in $M_0d$) and independent of $R$. And finally, we obtain a DHI for  {\it non-homogeneous} equations, which appears to be a completely novel result. The latter gives rise to some interesting phenomena related to the dimension and the right-hand side, which we discuss below.

As in the previous section,  here we restrict to \eqref{partoper}-\eqref{partoper_defM}, for readability. The constant $C_0$ depends on $n$, $\lambda$, $\Lambda$, $\alpha$.

\begin{thm}
\label{lgeu}
(i) For any weak solution $u$ of $-\mathcal{L}_0u=0$, $u>0$ in $B_R$,
\be\label{concllgeu}
\frac{d(x)|\nabla u(x)|}{u(x)}\le C_0\max\{1, M_0 d(x)\},\quad x\in B_R.
\ee

(ii) Assume $ q>\strut n=1$.   If $f\in L^q(-R,R)$ with $f\ge 0$, 
then for any weak solution of $-\mathcal{L}_0u=f$,  $u>0$ in $\Omega=(-R,R)$, inequality \eqref{concllgeu} is valid. \smallskip

(iii)  Assume  $q>\strut n\ge2$. If $f\in L^q(B_R)$ with $f\ge 0$, $f\not\equiv 0$, then
for any weak solution $u$ of $-\mathcal{L}_0u=f$, $u>0$ in $B_R$, we have the
estimate
\be\label{concllgeu2}
\frac{d(x)|\nabla u(x)|}{u(x)}\le C_0\Bigl(\max\{1, M_0d(x)\} + R^n e^{ C_0MR}
\frac{\|f\|_{L^q(B_R)}}{\|f\|_{L^1_d(B_R)}}\, d(x)^{1-\frac{n}{q}}\Bigr),\quad x\in B_R.
\ee
\end{thm}

\begin{rem}\label{remloggrad1d}
Thus, quite counter-intuitively, for $n=1$ the maximum of $d|u^\prime|/u$ is bounded above independently of $f$, for any positive solution of $-\mathcal{L}u=f$.
On the other hand, this is not so and \eqref{concllgeu} fails for  $n\ge2$ and $f\gneqq0$, as a consequence of Proposition~\ref{PropLoggradOptimality} below.
\end{rem}

\begin{rem}\label{remloggrad1d2}
It is easy to see that $|\nabla u|/u$ can be unbounded
without the positivity assumption $f\ge 0$, even for $n=1$,  cf.~\eqref{counterfpositive} below. The necessity to consider nonnegative right-hand sides could explain why it has not been observed before that logarithmic gradient bounds may hold for non-homogeneous equations.
\end{rem}

The proof of Theorem \ref{lgeu} differs considerably from those used in earlier works. It uses a contradiction argument, and is based on a logarithmic change of variable and  the doubling lemma from \cite{PQS}. That argument has some delicate points but is quite versatile since the only additional ingredient it requires are classical interior Schauder estimates, so it combines seamlessly with the Morel-Oswald estimate to yield part (iii), and with some additional functional analytic facts in one dimension, to yield part (ii) of the theorem.  We will also observe that parts (i) and (iii) of Theorem \ref{lgeu} admit a shorter proof at the cost of  using  the Harnack inequality for general operators (and some  careful rescalings).

We will show in Proposition~\ref{Optimality_Harnack} below  that in the homogeneous case $f=0$ the estimate~\eqref{concllgeu}
cannot be improved with respect to $M_0$. As for the non-homogeneous case -- for which the result above appears to be completely novel, even qualitatively -- we leave some open problems  with respect to the optimality of the second term in the right-hand side of~\eqref{concllgeu2}. We  discuss these in the next section.  

\medskip

The rest of the paper is organized as follows. The next section contains our main results for general operators and domains. In Section \ref{sec-prelim} we state several useful properties of the quantities $r_0,M$ and the uniformly local norms. Section \ref{proofHopf} contains the proof of the Morel-Oswald estimate, and Section \ref{proofloggrad} is devoted to the differential Harnack bound. The results on the sharpness of the estimates are proved in Sections \ref{sec-proofoptim}-\ref{sec-proofoptim2}, and at the end an appendix gives a few auxiliary results and proofs.

\needspace{4\baselineskip}
\section{Main results for general operators and domains }\label{sec-main}

{\cbl
\subsection{Notations} We always assume that}
 $$\bar\alpha\in (0,1], \quad n<q\le \infty, \quad\mbox{ and }\quad  0<\alpha< \min\{1-n/q,\bar\alpha\}.$$
 We denote with $c,C>0$ (possibly with indices) generic positive constants which may depend only on $n,\lambda,\Lambda, q,\alpha,\bar\alpha$, and may change from line to line.

 The domain $\Omega$ is  bounded, with $C^{1,\bar\alpha}$ boundary. The distance to the boundary of $\Omega$ is denoted by $d=d(x)=\mathrm{dist}(x,\partial\Omega)$. We set
$$D_0:={\rm diam}(\Omega),\qquad D:=\mbox{ geodesic diameter of }\Omega \;= \sup_{x,y\in \Omega} \inf_{\sigma\in\Sigma(x,y)} l(\sigma),$$
where $\Sigma(x,y)$ is the set of all paths in $\Omega$ connecting $x$ and $y$, and $l(\sigma)$ is the length of $\sigma$.
We next define a constant $r_\Omega$ which quantifies the well-known facts that each point of the boundary of a $C^{1,\bar\alpha}$-domain: (i) has a uniform neighborhood in which the domain can be ``flattened", and (ii) can be touched by a $C^{1,\bar\alpha}$-paraboloid with fixed size and opening.

Let $\rho_\Omega>0$ be the supremum of all
$r\le D_0$
such that, for each $\xi\in \partial\Omega$, there is a $C^{1,\bar\alpha}$-diffeomorphism $\Phi$ between the domain $(\Omega\cap B_{r}(\xi) - \xi)/r$ (resp. the surface $(\partial \Omega\cap B_{r}(\xi) - \xi)/r$) and $B_1^+=B_1(0)\cap\{x_n>0\}$ (resp. $B_1^0 = B_1(0)\cap\{x_n=0\}$) such that $D\Phi(0)=I$, $ |D\Phi| \le 2$, $ |D\Phi^{-1}| \le 2$, $[D\Phi]_\alpha\le1$. The existence of such $r>0$ is well known, see for instance~\cite{Liebe}.
Also, there exist $\bar\rho_\Omega, k_\Omega>0$ with the following property:
for each $\xi\in \partial\Omega$,
there exists an orthonormal coordinate system $y=(y',y_n)\in\R^{n-1}\times\R$ with origin $\xi$ and a $C^{1,\bar\alpha}$ function
$\varphi$ on $B'_0=\{|y'|<\bar\rho_\Omega\}$, such that,
setting $\Sigma_0=B'_0\times(-\bar\rho_\Omega,\bar\rho_\Omega)$, we have
$\Sigma_0\cap\Omega=\Sigma_0\cap \{y_n>\varphi(y')\}$,
$\Sigma_0\cap\partial\Omega=\Sigma_0\cap \{y_n=\varphi(y')\}$,
$\varphi(0)=\xi$, $D\varphi(0)=0$ and $[D\varphi]_{\bar\alpha}\le k_\Omega$; hence
\be\label{def_parab}
\bigl\{y=(y';y_n)\in \R^{n-1}\times\R:\ |y'|< \bar\rho_\Omega\ \hbox{and}\ k_\Omega |y'|^{1+\bar\alpha}<y_n< \bar\rho_\Omega\bigr\}\subset \Omega.\ee
For a very general result on these geometrical considerations see \cite{AM}, in particular Corollary~3.14 in that paper.
We set
\be\label{def_romega}
r_\Omega=\min\bigl(\rho_\Omega, \bar\rho_\Omega/4,(120 k_\Omega)^{-1/\bar\alpha})\bigr).
\ee
Note it is easy to see that if $\Omega=B_R$ then we can take $\bar\alpha=1$, and
$$
\mbox{if }\; \Omega=B_R\quad \mbox{then}\quad r_{\Omega}= c(n)R \; \mbox{ and }\; D=2R.
$$

We  denote by $\|\cdot\|_{L^q(\Omega)}$ the usual Lebesgue norm,
by $\|\cdot\|_{L^q_d(\Omega)}$ the Lebesgue norm weighted by $d$,
and by $[\cdot]_{\alpha,\Omega}$ the usual H\"older seminorm (bracket) on~$\Omega$.
We will use the following  uniformly local Lebesgue norm: for  $q\in[1,\infty]$, $r>0$,
\begin{equation}\label{deful}
\|h\|_{q,r,\Omega}:=\sup_{x\in \Omega} \|h\|_{L^q(\Omega\cap B_{{r}}(x))},
\quad h\in L^q(\Omega),
\end{equation}
(of course this is trivial for $q=\infty$,  $\|h\|_{\infty,r,\Omega}=\|h\|_{L^{\infty}(\Omega)}$ for any $r>0$), and the uniformly local H\"older bracket
\begin{equation}\label{defHbracket}
[\psi]_{\alpha, r,\Omega} := \sup_{x\in\overline{\Omega}}\ \sup_{y,z\in B_{r}(x)\cap \Omega} |y-z|^{-\alpha}|\psi(y)-\psi(z)|,
\quad \psi\in C(\overline\Omega),\quad  \alpha\in(0,1).
\end{equation}
As an advantage  over 
 the usual $L^q$-norms for finite $q$, uniformly local $L^q$-norms essentially measure
the local integrability features of a function and do not deteriorate in large domains for functions which do not decay for large $|x|$
(one may think of a periodic function).
These norms have been used in the study of global existence
of solutions of Navier-Stokes or parabolic equations, e.g.~in the classical papers \cite{Ka,GV} (for $r=1$)
and more recently in \cite{HOS,IsSa,MaTe} where the possibility to vary $r$ was also exploited.

Given an operator $\mathcal{L}$ as in \eqref{defdiv}
satisfying \eqref{hyp1}, \eqref{hyp2}, and recalling definitions \eqref{deful}, {\eqref{defHbracket},
we define the number $r_0 \,=r_0(\ld,\Omega) \in (0,r_\Omega]$ by
\begin{equation}\label{defr0}
r_0:= \sup\Bigl\{r>0:\: r\bigl(r^{-1}_\Omega+ [A]^{\frac{1}{\alpha}}_{\alpha, r, \Omega} + \|b_1\|_{L^\infty(\Omega)} + [b_1]^{\frac{1}{\alpha+1}}_{\alpha, r, \Omega}+\|b_2\|^{\beta_q}_{q,r,\Omega} + \|c\|^{\gamma_q}_{q, r, \Omega}\bigr)\le 1\Bigr\}.
\end{equation}
In other words, $r_0$ is the point where the increasing in $r$ function in the parentheses in \eqref{defr0} meets the decreasing function $1/r$. See also Proposition \ref{basicr0} and its proof below.

Our estimates depend on the quantity
\begin{equation}\label{defM}
M=M(\ld,\Omega) =
[A]_{\alpha,r_0,\Omega}^{\frac{1}{\alpha}} + \bigl\|b_1\bigr\|_{L^\infty(\Omega)} + [b_1]_{\alpha,r_0,\Omega}^{\frac{1}{1+\alpha}}+ \bigl\|b_2\bigr\|^{\beta_q}_{q,r_0,\Omega}+\|c\|^{\gamma_q}_{q,r_0,\Omega},
\end{equation}
$$\mathrm{where}\quad\beta_q=\frac{1}{1-{n}/{q}},\quad \gamma_q=\frac{1}{2-{n}/{q}} \qquad \left(\beta_\infty=1, \quad \gamma_\infty = {1}/{2}\right).$$

Since this definition may seem somewhat convoluted at first sight, we immediately note that if one does not aim at an entirely optimal estimate, it is easy to see (Remark~\ref{r01M}) that
 an upper bound for $M$ is given by
\begin{equation}\label{defMmajor}
M \le
1+[A]_{\alpha,1,\Omega}^{\frac{1}{\alpha}}+ [b_1]_{\alpha,1,\Omega}^{\frac{1}{1+\alpha}} + \bigl\|b_1\bigr\|_{L^\infty(\Omega)} + \bigl\|b_2\bigr\|^{\beta_q}_{q,1,\Omega}+\|c\|^{\gamma_q}_{q,1,\Omega},
\end{equation}
i.e., $r_0$ can be replaced by $1$ and $M$ by $M+1$,
which should be sufficient for most applications.
However, the optimality of the estimates and their scale invariance are guaranteed only with $r_0$ and $M$ as above, this is the level of difficulty of the problem when unbounded coefficients are allowed. We also warn the reader that $\|\cdot\|_{q,r_0,\Omega}$
does {\it not} behave as a norm when applied to the coefficients of the operator themselves,
due to the dependence of $r_0$ on these coefficients (see Proposition \ref{lemr0ul}
below.)

{\cbl
\subsection{General results of Morel-Oswald and DHI type} We give extensions to general domains and operator of the statements  in the previous section, which we gather in the following main theorem.}

\begin{thm}\label{thm4gen}
 Assume \eqref{hyp1}-\eqref{hyp2}
and let $r_0$, $M$  be defined in  \eqref{defr0}, \eqref{defM}.
\begin{itemize}

\item[(i)] If $u\in H^1(\Omega)$ is a weak supersolution
of $-\mathcal{L}u\ge f$ and $f,u\ge 0$ in $\Omega$, then
\be\label{conclHopfgen}
\frac{u(x)}{d(x)}
\ge
e^{-C_0(r^{-1}_\Omega+ M)D}D^{-n}
 \int_\Omega fd.
\ee

\item[(ii)] 1. For any weak solution $u$ of $-\mathcal{L}u=0$, $u>0$ in $\Omega$, we have the estimate
\be\label{conclHopfgendim1}
 \frac{d(x)|\nabla u(x)|}{u(x)}
\le C_0\max\bigl\{1, (r_\Omega^{-1}+M)d(x)\bigr\},\quad x\in \Omega.
\ee
2. If $n=1$, then \eqref{conclHopfgendim1} is valid for any positive weak solution of $-\mathcal{L}u=f\ge0$. Here a weak solution is a function
$u\in W^{1,p}(-R,R)$ with $p=\max(2,q/(q-1))$,
which satisfies the equation in the usual $H^1$ sense.

\smallskip 
\item[(iii)] If $f\in L^q(\Omega)$ with $f\ge 0$, $f\not\equiv 0$, then
for any weak solution $u$ of $-\mathcal{L}u=f$, $u>0$ in~$\Omega$, we have the
estimate
$$
 \frac{d(x)|\nabla u(x)|}{u(x)}
\le C_0\max\bigl\{1,(r_\Omega^{-1}+M)d(x)\bigr\} +
e^{C_0(r^{-1}_\Omega+ M)D} D^n
\frac{\|f\|_{L^q(\Omega)}}{\|f\|_{L^1_d(\Omega)}}\, d^{1-\frac{n}{q}}(x),\quad x\in \Omega.
$$
\end{itemize}
\end{thm}
We recall once more that $
r_{\Omega}= c(n)R$ and $D=2R$ when $\Omega=B_R$ so this theorem trivially reduces to those in the introduction, in the particular case $\mathcal{L}=\mathcal{L}_0$ considered in~\eqref{partoper} (observing also that $M=M_0$ from \eqref{partoper_defM} when $\mathcal{L}=\mathcal{L}_0$ has bounded coefficients).

\begin{rem}\label{remdefsol1}
In Theorem \ref{lgeu}(ii) we have slightly modified the definition of a weak solution, to guarantee that
the term $bu'$ (as well as $cu$) belongs to $L^1(-R,R)$, which  is not the case if $u$ is merely in $H^1(-R,R)$ when $q<2$.
Actually, if we take $u\in W^{1,\infty}(-R,R)$, the  statement in (ii) remains true if $b,c,f\in L^1(-R,R)$ instead of~$L^q(-R,R)$.
\end{rem}

The ramifications and methods of proof of Theorem \ref{thm4gen} were already discussed in the introduction. As we noted there, we will also use the global Harnack inequality from \cite{GSS}, in which we need to specify the constants. We give the statement here, because of its independent importance.

\begin{thm}\label{BHIoptim}
 Let $\Omega$ be a bounded $C^{1,\bar\alpha}$-domain of $\mathbb{R}^n$ with  geodesic diameter $D$, and let  $r_\Omega$ be defined in \eqref{def_romega}. Assume \eqref{hyp1}-\eqref{hyp2}.
There exist  constants $\epsilon, C_0>0$ depending only on $n,q,,\lambda,\Lambda,\alpha,\bar\alpha$, such that if $u\ge0$ in $\Omega$ is a weak solution of $-\mathcal{L}u\ge f $ in $\Omega$, for some $f\in L^q(\Omega)$,
then
\begin{equation}\label{sharpWBHI}
\left(\int_{\Omega} \left(\frac{u}{d}\right)^\epsilon\right)^{1/\epsilon} \le
D^{n/\epsilon}
e^{C_0(r^{-1}_\Omega+ M)D} \left( \inf_{\Omega} \frac{u}{d} +
D^{1-\frac{n}{q}} \|f\|_{L^q(\Omega)}\right),
\end{equation}
where $M=M(\ld,\Omega)$ is defined in \eqref{defM}.
If $-\mathcal{L}u= f $, $u\ge0$ in $\Omega$, $u=0$ on $\partial \Omega$, then
\begin{equation}\label{sharpBHI}
\sup_{\Omega} \frac{u}{d}\le
e^{C_0(r^{-1}_\Omega+ M)D}\inf_{\Omega} \frac{u}{d} +
e^{C_0(r^{-1}_\Omega+ M)D}D^{1-\frac{n}{q}} \|f\|_{L^q(\Omega)}.
\end{equation}
\end{thm}

The proof of Theorem \ref{BHIoptim} uses a Harnack chain argument similar to the one in the
proof of \cite[Theorem 2.1]{SS2}, replacing Theorem A there by \cite[Theorem 1.1]{GSS},
combined with a (sharp) estimate of the length of the chains in terms of the geodesic diameter (Proposition~\ref{geodes}).
For readers' convenience we list the technical differences and give a proof sketch of Theorem~\ref{BHIoptim} in the appendix.

Further, we discuss the sharpness of the constants in the above estimates. The statement (i) in Theorem \ref{thm4gen} is already covered by Theorem \ref{Prop-Optim-Hopf0} {\cbl (and \eqref{culinfty} below; see Section \ref{sec-proofoptim})}. The optimality of the constant in the global weak Harnack estimate \eqref{sharpWBHI} is trivial and follows similarly to the optimality of the interior estimate (see \cite[Remark~2.3]{SS2}),  by considering one-dimensional solutions of $u^{\prime\prime} +cu = 0$ or $u^{\prime\prime} +bu^\prime = 0$.  In the following proposition we prove that the constant in the differential Harnack estimate \eqref{concllgeu} (resp.~Theorem \ref{thm4gen} (ii))  cannot be improved, even if we restrict ourselves to solutions which vanish on the boundary. The same is valid for the constant in the global Harnack inequality \eqref{sharpBHI}.

\begin{prop}\label{Optimality_Harnack}
Let $R>0$.
There exists a constant
$\lambda_0=\lambda_0(n)>0$,  sequences $b_i, c_i\in C^\infty(\overline{B_R})$,
and a classical solution $u_i>0$ of either of the following problems
\be \label{optimLinfty1ci}
-\ld_iu_i:=-\Delta u_i+c_iu_i=0, \quad \mbox{in } B_R,\qquad
u_i=0, \quad \mbox{on } \partial B_R,
\ee
\be \label{optimLinfty1bi}
-\ld_iu_i:=-\Delta u_i + b_i\cdot\nabla u_i - \lambda_0R^{-2} u_i=0, \quad \mbox{in } B_R,\qquad
u_i=0 \quad \mbox{on } \partial B_R,
\ee
such that $\|b_i\|_{q,r_i,B_R}$, $\|c_i\|_{q,r_i,B_R}\to\infty$ with $r_i=r_0(\ld_i,B_R)$,
\be \label{optimLinfty23}
\frac{\sup_{B_R}(u_i/d)}{\inf_{B_R}(u_i/d)}\ge Ce^{CM_iR}, \qquad \mbox{and} \qquad
\frac{|\nabla u_i|}{u_i}\ge CM_i\quad\hbox{on $\{|x|=R/2\}$},
\ee
where $C=C(n)>0$ and $M_i=\|c_i\|^{\gamma_q}_{q,r_i,B_R}$, resp.~$M_i=\|b_i\|^{\beta_q}_{q,r_i,B_R}$.
\end{prop}

\begin{rem}\label{remoptloggrad}
We observe that, under the assumptions of Theorem \ref{lgeu}(i), we always have
$\frac{d(x)|\nabla u(x)|}{u(x)}\to 1$ as $|x|\to R$ which, together with
\eqref{optimLinfty23},
shows that the constant in the differential Harnack  estimate \eqref{concllgeu} cannot be improved.
\end{rem}

\begin{rem} We  will  see
that the exponential constant in front of the right-hand side $f$ in the Harnack inequality \eqref{sharpBHI}
  cannot be improved either, Proposition~\ref{Optimality_Harnack2}.
\end{rem}

Even though we focus here  on the  dependence of the constants in the coefficients and the size of the domain, it is also worth recalling  and discussing the much more classical question of the optimality of the Lebesgue norms which appear in the estimates. It is very well known that  $C^1$-estimates fail for coefficients or right-hand sides which are only in $L^n$ instead of $L^q$, $q>n$ (for instance $u(x)=|x|\log|\log|x||$ solves $\Delta u = f\in L^n$, and variations and smooth approximations of that function can be used to violate our estimates for right-hand sides in $L^n$).
Further, the norm $\|f\|_{L^1_d}$ in the right-hand side of the Morel-Oswald inequality  cannot be replaced by $\|f\|_{L^p_d}$ nor by $\|f\|_{L^1_{d^\epsilon}}$, for $p>1$ or $\epsilon<1$ (this follows from classical estimates for the Green function of the Laplacian, see Remark~\ref{Hopf-optL1d}). 
However, it is much less obvious that the norm $\|f\|_{L^1_d}$ in the denominator in the right-hand side of the differential Harnack estimate
\eqref{concllgeu2} or Theorem \ref{thm4gen} (iii) cannot be improved. 
 We can prove  only in dimension two (Proposition \ref{PropLoggradOptimality} below) that \eqref{concllgeu2} fails if $\|f\|_{L^1_d}$
is replaced by $\|f\|_{L^p_d}$, $p>1$, and $q$ is strictly larger but close to $n$.
In dimension $n\ge 3$, we have the same conclusion on the failure of \eqref{concllgeu2} only for $p>n/2$, so that the optimality
of the $L^1_d$ norm in \eqref{concllgeu2} remains unclear in this case. We also do not know whether the exponential constant in front of the second term in the right-hand side of \eqref{concllgeu2} is sharp.

\begin{rem}\label{remcoefopt}
In the end, we observe that solutions of \eqref{defdiv} are known to be continuously differentiable under weaker hypotheses on the  coefficients or the domain. For instance, we could assume that $A,b_1$ have Dini mean oscillation, and/or that $\partial\Omega$ is $C^{1,Dini}$ (see \cite{DEK}). We expect our results to have natural extensions to this case, however the form of the quantities $M$ and $r_\Omega$ would probably become rather complicated to write explicitly. The same remark applies to the integrability of the lower-order coefficients, which might be assumed to belong to Orlicz, Kato or Lorentz spaces that are intermediate between $L^n$ and $L^q$, $q>n$, and still ensure $C^1$-regularity.

Furthermore, for the Morel-Oswald inequality it is actually sufficient to assume the  regularity required in \eqref{hyp2} only in some neighbourhood of $\partial\Omega$ instead of the whole $\Omega$, while only assuming \eqref{hyp1} and $b_1,b_2\in L^q(\Omega)$, $c,f\in L^{q/2}(\Omega)$, $q>n$. Indeed, the global Harnack inequality in \cite{GSS} is already written under this hypothesis, and the proof of Theorem \ref{thm4gen}(i) requires $C^1$ regularity only in a neighborhood of the boundary. We leave the easy extension of that proof and the quantity $M$ to the interested reader.  

Finally, we recall that the Hopf-Oleinik lemma, and a fortiori the Morel-Oswald estimate, fail if we take $\alpha=0$ or $\bar \alpha =0$ in the above (Remark \ref{Hopf-optC1} below). 
\end{rem}

\section{Auxiliary results} \label{sec-prelim}

\subsection{Properties associated with uniformly local norms}

In this subsection we give some useful properties of the quantities $r_0, M$,  and the
associated uniformly local norms. Their proofs can be found in the appendix.

\begin{prop}\label{basicr0}  Recalling definitions \eqref{def_romega}-\eqref{defM},  we have
\begin{equation}\label{relMr0}
r_0^{-1}=r^{-1}_\Omega+ M,\qquad \mbox{hence }\; \frac{1}{2}\min\bigl(r_\Omega,M^{-1}\bigr)\le r_0\le\min\bigl(r_\Omega,M^{-1}\bigr).
\end{equation}
\end{prop}

\begin{rem} \label{r01M}
As mentioned above,  $M$ satisfies the upper bound \eqref{defMmajor}, so that 
all  results in Sections \ref{sec-intro}-\ref{sec-main} remain valid if $r_0$ is replaced by $1$ and $M$ by $M+1$.
Indeed, if $r_0\le 1$, then  $\|\cdot\|_{q,r_0,\Omega}\le \|\cdot\|_{q,1,\Omega}$
and $[\cdot]_{\alpha, r_0, \Omega}\le [\cdot]_{\alpha, 1, \Omega}$,
whereas if $r_0>1$, then $M<1$ owing to \eqref{relMr0}.
\end{rem}

We give a simple monotonicity property for general domains.

\begin{prop}\label{scaleinvtilde0}
Let $\omega\subset\Omega$ be bounded domains with $C^{1,\bar\alpha}$ boundaries,
$\mathcal{L}$ be an operator as in \eqref{defdiv} satisfying \eqref{hyp1}, \eqref{hyp2},
 and let $\theta\in(0,1]$. There exists $C(n,q,\alpha,\bar\alpha,\theta)>0$ such that,
 if $r_\Omega\ge\theta r_\omega$, then
$M(\ld,\omega)\le C(n,q,\alpha,\theta)M(\ld,\Omega)$.
\end{prop}

We now consider the case of balls.
It is immediate that, if $u$ is a solution of $\mathcal{L}u=f$ in $\Omega=B_R$, then $\tilde u(x)=u(Rx)$ is a solution of
$\tilde{\mathcal{L}}u=\tilde f$ in $B_1$,
where $\tilde f$ and the  coefficients 
of $\tilde{\mathcal{L}}$ are given by
\begin{equation}\label{scaledcoeff}
\tilde A(x)=A(Rx),\quad \tilde b_i(x)=Rb_i(Rx),\quad \tilde c(x)=R^2c(Rx),\quad \tilde f(x)=R^2f(Rx).
\end{equation}
In the next proposition, after giving a basic monotonicity property of the quantity $M$ with respect to $R$,
we show that
$r_0, M$, as well as each of the (uniformly local)
norms and seminorms appearing in $M$
enjoy some natural invariance properties with respect to the scaling transformation \eqref{scaledcoeff}.
As a consequence, all main estimates in Section~\ref{sec-main} in the case $\Omega=B_R$
are scale-invariant with respect to $R$, and
in particular the general case $R>0$ can be reduced to the case $R=1$.

\begin{prop}\label{scaleinvtilde}
Let $R>0$ and assume \eqref{hyp1}-\eqref{hyp2} with $\Omega=B_R$.
\smallskip

{(i) We have $M(\mathcal{L},B_\rho)\le M(\mathcal{L},B_R)$ for all $\rho\in(0,R)$.
\smallskip

(ii) Let $f\in L^q(B_R)$} and let $\tilde f$ and the coefficients of the operator $\tilde{\mathcal{L}}$ be given by \eqref{scaledcoeff}.
Let $\tilde r_0$ be defined by \eqref{defr0} with $\mathcal{L},B_R$
replaced by $\tilde{\mathcal{L}},B_1$.
Then
\begin{equation}\label{scaleinvtilde2}
\tilde r_0=r_0/R,
\end{equation}
\begin{equation}\label{scaleinvtilde2b}
[\tilde A]_{\alpha, \tilde r_0, B_1}= R^{\alpha}[A]_{\alpha, r_0, B_R},\,
[\tilde b_1]_{\alpha, \tilde r_0, B_1}= R^{1+\alpha}[b_1]_{\alpha, r_0, B_R},\,
 \|\tilde b_1\|_{L^\infty(B_1)}=R\|b_1\|_{L^\infty(B_R)},
\end{equation}
\begin{equation}\label{scaleinvtilde4}
\|\tilde b_2\|_{q,\tilde r_0,B_1}=R^{1-\frac{n}{q}}\|b_2\|_{q,r_0,B_R},\
\|\tilde c\|_{q,\tilde r_0,B_1}=R^{2-\frac{n}{q}}\|c\|_{q,r_0,B_R},\
\|\tilde f\|_{q,\tilde r_0,B_1}=R^{2-\frac{n}{q}}\|f\|_{q,r_0,B_R}.
\end{equation}
Consequently,
\begin{equation}\label{scaleinvtilde5}
M(\tilde{\mathcal{L}},B_1)=RM(\mathcal{L},B_R).
\end{equation}
\end{prop}

We next
consider the properties of the quantity $\|\cdot\|_{q,r_0,B_R}$
(for $q\in(n,\infty)$)
when applied to the coefficients of the operators themselves:
in this situation it does not obey the scaling of a norm
but a {\it nonlinear} scaling (due to the dependence of $r_0$ on the coefficients of the operators).
Here we restrict to a smaller class of operators, for simplicity and since it will be enough for our purposes.

\begin{prop} \label{lemr0ul}
Let $R>0$ and $b, c\in L^\infty(B_R)$.
\smallskip

(i) Let $\mathcal{L}=\Delta+b(x)\cdot\nabla +c(x)$.
Then we have
\be\label{culinfty}
\|b\|_{q,r_0,B_R} \le C(n) \|b\|^{1-n/q}_{L^\infty(B_R)}\qquad\mbox{and}\qquad
\|c\|_{q,r_0,B_R} \le C(n) \|c\|^{1-n/2q}_{L^\infty(B_R)}.
\ee

(ii) Consider the family of operators
$\mathcal{L}_\lambda=\Delta+\lambda b(x)\cdot\nabla$
(resp., $\Delta+\lambda c(x)$) for $\lambda>0$,
and
set $r_\lambda:=r_0(\mathcal{L}_\lambda,B_R)$ {\rm(}cf.~\eqref{defr0}{\rm)}.
Then we have
\be\label{culinfty2}
C_1(n)\lambda^{1-n/q} \|b\|_{q,r_1,B_R}\le \|\lambda b\|_{q,r_\lambda,B_R} \le
C_2(n) \lambda^{1-n/q}\|b\|^{1-n/q}_{L^\infty(B_R)} ,\quad \lambda\ge 1
\ee
$$({\rm resp.,}\quad C_1(n)\lambda^{1-n/2q} \|c\|_{q,{r_1},B_R}
\le \|\lambda c\|_{q,{r_\lambda},B_R} \le
C_2(n) \lambda^{1-n/2q}\|c\|^{1-n/2q}_{L^\infty(B_R)} ,\quad \lambda\ge 1).$$
Moreover the function $\lambda\mapsto
\|\lambda b\|_{q,r_\lambda,B_R}$
is continuous on $(0,\infty)$.
\end{prop}

\section{Proof of the Morel-Oswald estimate (Theorems~\ref{OptimizedHopf} and \ref{thm4gen}(i))}
\label{proofHopf}

Let $\Omega$ be a domain with $C^{1,\bar\alpha}$ boundary and $\tilde \ld$ be an operator of the form \eqref{defdiv} defined in $\Omega$, with coefficients $\tilde A$, $\tilde b_i$, $\tilde c$ satisfying \eqref{hyp1}-\eqref{hyp2}. For  $B_r=B_r(x_0)$, $x_0\in \partial\Omega$, we denote $B_r^+=B_r\cap\Omega$, $B_r^0=B_r\cap\partial\Omega$.
 In view of the proof of Theorem~\ref{thm4gen}(i), we prepare
the following lemma, which provides quantitative interior and boundary Morel-Oswald type estimates
for normalized domains and  coefficients with small Lebesgue norms.

\begin{lem} \label{LemOptimizedHopf2}
 Assume that either (i) $\omega= B_2\subset\Omega$ or (ii) $\omega$ is a $C^{1,\bar\alpha}$ domain  such that   $B_1^+\subset\omega\subset B_2^+$ and the $C^{1,\bar\alpha}$ norm of $\partial\omega$ is bounded by $C=C(n,\bar\alpha)>0$. There exist constants $\ep_0, c_0>0$ depending on $n,\lambda,\Lambda, q, \alpha, \bar\alpha$, such that if
\be\label{hyplambda1b}
[\tilde A]_{\alpha,\omega}+\|\tilde b_1\|_{C^\alpha(\omega)}\le1, \qquad  \|\tilde b_1\|_{L^q(\omega)}
+ \|\tilde b_2\|_{L^q(\omega)}+\|\tilde c\|_{L^{q}(\omega)}\le \ep_0,
\ee
and $v\ge0$ is a weak supersolution of $-\tilde{\mathcal{L}}v\ge g \ge0$ in $\omega$ for some $g\in L^q(\omega)$, then
$$\inf_{\omega} \frac{v}{d}\ge c_0 \int_{\omega}g d,\qquad d(x)=\mathrm{dist}(x,\partial\omega).$$
\end{lem}

\begin{proof}[Proof of Lemma~\ref{LemOptimizedHopf2}] We will give the proof of statement (ii) (the proof of the interior statement (i) is simpler and goes the same way).

We  can assume that $\lambda_1(-\tilde{\mathcal{L}},\omega)>0$ and that $v>0$ is a solution of $-\tilde{\mathcal{L}}v= g$ in $\omega$ with $v=0$ on $\partial \omega$, and so $v\in C^{1,\alpha}(\overline{\omega})$. Indeed, the positivity of the first eigenvalue follows from Proposition \ref{lowerbdeig} below and an appropriate choice of $\ep_0$, and  then  it is sufficient to replace $v$ by the solution of the Dirichlet problem $-\tilde{\mathcal{L}}\hat v= g$ in $\omega$ with $\hat v=0$ on $\partial \omega$, and to observe that $v\ge \hat v$ by the maximum principle (the maximum principle and the solvability of the Dirichlet problem are available thanks to $\lambda_1(-\tilde{\mathcal{L}},\omega)>0$, see the appendix).

It is easy to see that  there is a universal $\rho>0$ and a point $\xi$ such that the ball $B_{4\rho}(\xi)\subset B_1^+$ (for instance, in the coordinate system $y=(y^\prime,y_n)$ used in the beginning of Section \ref{sec-main}, we can take $\xi=(0,\rho_\Omega/2)$, since by \eqref{def_parab}-\eqref{def_romega} we have $\{y\::\: |y^\prime|<y_n, \rho_\Omega/4<y_n<3\rho_\Omega/4\}\subset\Omega$).  Set $B_\rho=B_\rho(\xi)$.

In the following we will use the {\it dual operator} of $\ld$
$$
 \ld^*v
= \mathrm{div}(\tilde A^T \nabla v - \tilde b_2v) - \tilde b_1 \nabla v + \tilde cv,
$$
such that for each $u,v\in H^1_0(\Omega)$ we have $
\int_\Omega (-\ldu)v = \int_\Omega u(-\ld^*v)$  
 in the weak sense.

By \eqref{hyplambda1b} and Proposition~\ref{lowerbdeig}  we also  know that we can choose $\epsilon_0$ such that
\be\label{upperlam2}
0<\lambda_1(-\tilde{\mathcal{L}^*},\omega)=\lambda_1(-\tilde{\mathcal{L}},\omega)\le \lambda_1(-\tilde{\mathcal{L}},B_{2\rho})\le C_0.\ee

{\bf Step 1.}
Set $\ld_0u= \mathrm{div}(A(x)\nabla u)$, with $\|A\|_{C^\alpha(\omega)}\le \Lambda+1$. There exists $c_0>0$ depending on $n,\lambda,\Lambda,\alpha, \bar\alpha$ for which the solution of the auxiliary problem
$$
\left\{\hskip 2mm\begin{aligned}
-\ld_0w_0&=\chi_{B_\rho} &\quad&x\in \omega,\\
w_0&=0, &\quad&x\in \partial \omega,
\end{aligned}\right.
\qquad
\mbox{is such that}\qquad
\inf_{\omega} \frac{w_0}{d} \ge c_0.
$$

\noindent{\it Proof}. We  consider the first eigenvalue and eigenfunction  $\lambda_1,\varphi_1>0$, $\varphi_1\in C^{1,\alpha}(\overline{\omega})$,
of $\ld_0^*= \mathrm{div}(A^TD\cdot)$ in $B_{2\rho}$,
normalized by $\max_{B_{2\rho}}\varphi_1=1$. Clearly $0<\lambda_1\le C_0$.
By the global Harnack inequality (cf.~Theorem \ref{BHIoptim}) we have
$\inf_{B_{2\rho}}\frac{\varphi_1}{d_0}\ge c_0 \sup_{B_{2\rho}}\frac{\varphi_1}{d_0}\ge c_0 \sup_{B_{2\rho}}\varphi_1$,
where $d_0(x)={\rm dist}(x,\partial B_{2\rho})$, hence in particular
$$
\inf_{B_{\rho}}\varphi_1\ge c_0.
$$

Since $-\ld_0w_0 \ge 0$, $w_0>0$ in $\omega$,   by  the global weak Harnack inequality \eqref{sharpWBHI}
\be\label{intharn1}\inf_{\omega} \frac{w_0}{d}\ge c_0 \Bigl\| \frac{w_0}{d}\Bigr\|_{L^\eps(\omega)}\ge c_0 \Bigl\|
\frac{w_0}{d}\Bigr\|_{L^\eps(B_{2\rho})}\ge c_0 \|{w_0}\|_{L^\eps(B_{2\rho})}\ge c_0 \inf_{ B_{2\rho}}w_0, 
\ee
whereas by the interior weak Harnack inequality
\cite[Theorem 8.18]{GT},
\be\label{intharn2}
\inf_{B_{2\rho}}w_0\ge c_0 \int_{B_{2\rho}} w_0,
\ee
and by using \eqref{upperlam2} and standard integration by parts
\be\label{HopfIPP}
\begin{aligned}
\int_{B_{2\rho}} w_0
&\ge \int_{B_{2\rho}} w_0\varphi_1= \frac{1}{\lambda_1} \int_{B_{2\rho}} w_0(-\ld_0^*\varphi_1)\\
&= \frac{1}{\lambda_1} \int_{B_{2\rho}} (-\ld_0w_0)\varphi_1 -\frac{1}{\lambda_1}\int_{\partial B_{2\rho}}  \nu\cdot A(x) w_0
\nabla\varphi_1\, d\sigma\\
&\ge c_0 \int_{B_{2\rho}} \chi_{B_\rho}\varphi_1 = c_0 \int_{B_{\rho}} \varphi_1 \ge  c_0 \inf_{B_{\rho}}\varphi_1\ge c_0,
\end{aligned}
\ee
since $\nabla\varphi_1=(\partial_\nu \varphi_1)\nu$ and $\partial_\nu\varphi_1\le 0$, $\nu\cdot A(x)\nu\ge 0$ on $\partial B_{2\rho}$.
Combining this with \eqref{intharn1}-\eqref{intharn2} gives the claim.

\smallskip

{\bf Step 2.}
Set $\ld_1u= \mathrm{div}(A(x)\nabla u)+b(x)\cdot \nabla u$, 
$b\in L^q(\omega)$.  There exist constants $c_1>0$, $\ep_1\in(0,1)$ depending on $n,\lambda,\Lambda, q, \alpha,\bar\alpha$ such that if $\|A\|_{C^\alpha(\omega)}\le \Lambda+1$ and $\|b\|_{L^q(\omega)}\le \ep_1$ then the solution of the auxiliary problem
$$
\left\{\hskip 2mm\begin{aligned}
-\ld_1w_1&=\chi_{B_\rho} &\quad&x\in \omega,\\
w_1&=0, &\quad&x\in \partial \omega,
\end{aligned}\right.
\qquad
\mbox{is such that}\qquad
\inf_{\omega} \frac{w_1}{d} \ge c_1.
$$

\noindent{\it Proof}. By standard $L^\infty$ and $C^1$ estimates (see \cite[Theorem 8.16]{GT}, \cite[Corollary 8.36]{GT},  the remark at the end of \cite[Section~8.10]{GT},  \cite[Theorem 5.5.5']{Mo}) if $\|b\|_{L^q(\omega)}\le 1$ any solution of $-\ld_iz_i =  h$ in $\omega$, $z_i=0$ on $\partial\omega$ ($i=0,1$, $\ld_0$ is the operator from Step 1), $h\in L^q(\omega)$, is such that for some $C_0$ depending on $n,\lambda,\Lambda, q, \alpha,\bar\alpha$,
$$
\|z_i\|_{C^1(\omega)}\le C_0\| h\|_{L^q(\omega)}.
$$

Now notice that $w_1 = w_0-z_0$, where $w_0$ is the function from Step 1, and $z_0$ solves
$$-\ld_0z_0 = - b\cdot \nabla w_1=:h\quad\mbox{ in }\quad \omega.$$
Since $\|h\|_{L^q(\omega)}\le C_0\ep_1$, by choosing $\ep_1$ small enough we have
$\|z_0\|_{C^1(\omega)}\le  c_0/2$, where $ c_0$ is the constant from Step 1. Thus
$$
\inf_{\omega} \frac{w_1}{d} \ge \inf_{\omega} \frac{w_0}{d} - \sup_{\omega} \frac{z_0}{d}\ge  c_0/2.
$$

\smallskip

{\bf Step 3.}
There exist constants $\epsilon_0, c_0>0$ depending on $n,\lambda,\Lambda, q, \alpha,\bar\alpha$, such that if \eqref{hyplambda1b} holds then the solution of the auxiliary problem
$$
\left\{\hskip 2mm\begin{aligned}
-\tilde \ld^*w&=\chi_{B_\rho} &\quad&x\in \omega,\\
w&=0, &\quad&x\in \partial \omega,
\end{aligned}\right.
\qquad
\mbox{is such that}\qquad
\inf_{\omega} \frac{w}{d} \ge c_0.
$$

\noindent{\it Proof}. We recall that $\tilde \ld^*w = \mathrm{div}(\tilde A^T \nabla w - \tilde b
_2w) - \tilde b_1 \nabla w + \tilde cw$. Let $\psi\in H^1_0(\omega)$ be the solution of the Dirichlet problem
$$
\tilde \ld^*\psi = -div(\tilde b_2) + \tilde c
$$
(this problem has a unique solution since $\lambda_1(-\tilde \ld^*,\omega)>0$, cf.~\eqref{upperlam2}).
By \eqref{hyplambda1b} and the generalized maximum principle \cite[Theorem 8.16]{GT} we have \be\label{estlinfpsi}
\|\psi\|_{L^\infty(\omega)}\le C\eps_0.\ee
We can also write
$$
\mathrm{div}(\tilde A^T \nabla\psi) - \tilde b_1 \nabla \psi  = -div(\tilde b_2(1-\psi)) + \tilde c(1-\psi)
$$
and apply  Morrey's $W^{1,q}$-regularity estimate (see \cite[Section~5.5]{Mo}) to this equation.
Together with  \eqref{estlinfpsi} and Sobolev embeddings  this gives
$$
\|\psi\|_{C^\alpha(\omega)}\le C\|\psi\|_{W^{1,q}(\omega)}\le
C(\|\psi\|_{L^\infty(\omega)} + \|\tilde b_2(1-\psi)\|_{L^q(\omega)} + \|\tilde c(1-\psi)\|_{L^q(\omega)})\le C_1\ep_0,
$$
provided  $\ep_0\le 1$.
We further diminish $\ep_0$ so that $C_1\ep_0<1/2$. Then the function $\phi=1-\psi$ is such that
\be\label{propfi}
\tilde \ld^*\phi= 0 \;\mbox{ in } \omega,\qquad 1/2\le\phi\le3/2 \;\mbox{ in } \omega\qquad \mbox{and} \qquad [\phi]_{\alpha,\omega}\le C\|\nabla \phi\|_{L^q(\omega)}\le C\ep_0.
\ee

Set $w_1= w/\phi$. It is easy to compute that $\tilde \ld^*w = \tilde \ld^*(\phi w_1) = \hat \ld (w_1)$, where $\hat \ld$ is the operator of the form \eqref{defdiv} whose coefficients are given by
$$
\hat A = \phi \tilde A^T, \quad \hat b_1 = \tilde A^T \nabla\phi - \tilde b_2 \phi, \quad
\hat b_2 =  -\tilde b_1\phi,
\quad \hat c =- \tilde b_1 \nabla \phi + \tilde c\phi.
$$
Since div$(\hat b_1) + \hat c = \tilde \ld^*\phi= 0$ in the sense of distributions in $\omega$, we see that the equation $-\hat \ld(w_1) = -\tilde \ld^*w=\chi_{B_\rho}$  is equivalent in the weak sense to
$-\ld_1w_1 = \chi_{B_\rho}$ where
$$
\ld_1 w_1 = \mathrm{div}(\phi \tilde A^T \nabla w_1) +(  \tilde A^T \nabla \phi -(\tilde b_1+\tilde b_2)\phi)\nabla w_1$$
is an operator to which Step 2 applies, thanks to \eqref{propfi} and an appropriate choice of $\ep_0$. So $w_1\ge cd$ and by using \eqref{propfi} again we get $w\ge (c/2)d$ in $\omega$.

\smallskip
{\bf Step 4.} We claim that
\be \label{HopfAuxClaim2}
\inf_{B_\rho}v\ge c_0  \int_{\omega} g d.
\ee
Indeed, by the interior weak Harnack inequality (see \cite[Theorem 8.18]{GT} and remark on \cite[p.209]{GT}) and assumption \eqref{hyplambda1b} with $\ep_0\le1$, we have $\inf_{B_\rho}v\ge c_0\int_{B_\rho} v$. Therefore, by the duality, if $w$ is the function from the previous step
$$\inf_{B_\rho}v\ge c_0\int_{\omega} v\chi_{B_\rho}=c_0\int_{\omega} v(-\tilde{\mathcal{L}}^*w)
= c_0\int_{\omega} w(-\tilde{\mathcal{L}}v)\ge c_0\int_{\omega} gw.$$
Step 3 then gives \eqref{HopfAuxClaim2}.

\smallskip

{\bf Step 5.} We introduce the auxiliary problem:
$$
-\tilde{\mathcal{L}}z=0 \mbox{ in } \omega\setminus B_\rho,\qquad
z=0 \mbox{ on } \partial \omega,\qquad
z=1 \mbox{ on } \partial B_\rho.
$$
Note that this problem is solvable owing to $\lambda_1(-\tilde{\mathcal{L}},\omega\setminus B_\rho)\ge\lambda_1(-\tilde{\mathcal{L}},\omega)>0$, and then $z>0$ by the maximum principle.
We claim that
\be \label{HopfAux6}
\inf_{\omega\setminus B_\rho} \frac{z}{d} \ge c_0.
\ee

Indeed, by \eqref{hyplambda1b} and the H\"older estimate (cf.~\cite[Theorem 8.29]{GT} and the remark at the end of \cite[Section 8.10]{GT}), there exist
$c_0, C_0>0$, and $\delta_0\in(0,\rho)$, depending only on $n,q,\lambda,\Lambda,\alpha, \bar \alpha$, such that (using $z=1$ on $\partial B_\rho$)
\be \label{HopfAux4}
|z(x)-1|\le C_0\delta_0^\alpha, \quad\mbox{hence} \quad z(x)\ge 1/2\ge c_0 d(x),\quad \mbox{for }\; x\in B_{\rho+\delta_0}\setminus B_\rho,
\ee
so  by the global  Harnack inequality,  with $\hat d(x)={\rm dist}(x,\partial (\omega\setminus B_{\rho}))$,
\be \label{HopfAux5}
\inf_{\omega\setminus B_{\rho}} \frac{z}{\hat d} \ge c_0 \sup_{\omega\setminus B_{\rho}} \frac{z}{\hat d}\ge c_0 \sup_{ B_{2\rho}\setminus B_{\rho}} {z}\ge c_0/2.
\ee
Combining \eqref{HopfAux4} and \eqref{HopfAux5} proves the claim.

\smallskip
{\bf Step 6.} Conclusion.
Set $Z=v-c_0(\int_{ \omega} g d)z$.
We have $-\tilde{\mathcal{L}}Z\ge 0$ in $\omega\setminus B_\rho$, $Z=0$ on $\partial \omega$ and $Z\ge 0$ on $\partial B_\rho$
in view of \eqref{HopfAuxClaim2}.
By the maximum principle (which applies owing again to $\lambda_1(-\tilde{\mathcal{L}},\omega)>0$), we deduce that
$Z\ge 0$ in $\omega\setminus B_\rho$. Therefore, by \eqref{HopfAux6},
$$\inf_{\omega\setminus B_\rho} \frac{v}{d} \ge c_0 \int_{\omega} g d.$$
This combined with \eqref{HopfAuxClaim2} completes the proof of Lemma~\ref{LemOptimizedHopf2}.
\end{proof}

The next simple lemma will permit us to cover the domain $\Omega$
with small balls such that a suitably rescaled version of the operator in each of these balls
satisfies the assumptions of Lemma~\ref{LemOptimizedHopf2}.
\begin{lem} \label{LemOptimizedHopf3}
Let $A, b_1, b_2, c$ satisfy the assumptions of Theorem~\ref{thm4gen}(i).
Let $x_0\in \overline\Omega$
and $r_0$ be the number from \eqref{defr0}. There exists $\delta\in(0,1/2)$ depending only on $n, \lambda, \Lambda, q, \alpha, \bar\alpha$, such that
the (rescaled) operator
\be\label{def-A-resc1}
\tilde{\mathcal{L}} v=\tilde{\mathcal{L}}_\delta v=\mathrm{div}(\tilde A(y)\nabla v+\tilde b_1(y) v)+\tilde b_2(y)\cdot \nabla v+\tilde c(y)v,
\quad y\in  \omega,
\ee
where $x=x_0+ry$, $$r=\delta r_0,\qquad\omega=\omega_\delta= r^{-1}(\Omega-x_0)\cap B_2,$$
\be\label{def-A-resc2}
\tilde A(y)=A(x_0+ry), \quad \tilde b_i(y)=rb_i(x_0+ry),\quad \tilde c(y)=r^2c(x_0+ry).
\ee
is such that \eqref{hyplambda1b} holds.
 If, moreover, $x_0\in\partial\Omega$ then
there exists a $C^{1,\bar\alpha}$ domain $\tilde\omega$
 with norm bounded by~$C=C(n,\bar\alpha)>0$, such that $r^{-1}(\Omega-x_0)\cap B_{3/2}\subset\tilde\omega\subset r^{-1}(\Omega-x_0)\cap B_2$.
\end{lem}

\begin{proof} The first statement follows  by the following simple  computation: fix any $x_0\in \overline\Omega$, $r=\delta r_0\in (0,r_0/2]$,
{where $r_0$ is the number from \eqref{defr0},} and let $v(y)=u(x_0+ry)$. We note that if $u$ satisfies $\ld u=(\le,\ge)f$ in $\Omega$ then
the function $v$ satisfies
$$\tilde{\mathcal{L}}v=(\le,\ge)\tilde{f}\quad\hbox{ in $\omega:=B_2\cap r^{-1}(\Omega-x_0)$,}$$
where the coefficients of the modified operator $\tilde{\mathcal{L}}$ are $\tilde A(y)= A(x_0+r y)$,
$\tilde b_i(y)=r
b_i(x_0+r y)$, $\tilde c(x)=r^2c(x_0+ry)$, and $\tilde{f}(x)=r^2 f(x_0+ry)$. We compute
\be\label{rescaledLq}
\begin{aligned}
\|\tilde b_2\|_{L^q(\omega)}
&=r\Bigl(\int_{\omega} |b_2(x_0+ry)|^q\,dy\Bigr)^{1/q}\\
&\le r^{1-n/q}\Bigl(\int_{\Omega\cap B_{r}(x_0)} |b_2(x)|^q\,dx\Bigr)^{1/q}
\le r^{1-n/q}\|b_2\|_{q,r_0,\Omega}
\end{aligned}
\ee
and similarly  $\|\tilde c\|_{L^q(\omega)} \le r^{2-n/q}\|c\|_{q,r_0,\Omega}$,
$\|\tilde f\|_{L^q(\omega)} \le r^{2-n/q}\|f\|_{q,r_0,\Omega}$, $[\tilde A]_{\alpha,\omega}\le r^\alpha [\tilde A]_{\alpha,r_0,\Omega}$. Hence, by the definition of $r_0$, for $\delta\le1$,
$$[\tilde A]_{\alpha, 1, \omega}\le 1,\quad \|\tilde b_1\|_{L^\infty(\omega)}\le 1,\quad  [\tilde b_1]_{\alpha, 1, \omega}\le 1, \quad\|\tilde b_2\|_{L^q(\omega)}\le 1,\quad \|\tilde c\|_{L^q(\omega)}\le 1$$
so by further diminishing $\delta$, from \eqref{rescaledLq} we get \eqref{hyplambda1b}.

The second statement is clear from the definition of $r_\Omega$ at the beginning of Section~\ref{sec-main}
and the fact that $r_0\le r_\Omega$.
\end{proof}

We are now in a position to prove Theorem~\ref{thm4gen}(i), of which Theorem \ref{OptimizedHopf} is a special case.

\begin{proof}[Proof of Theorem~\ref{thm4gen}(i)]
Since $f\ge 0$, the  weak
global Harnack inequality \eqref{sharpWBHI} gives
\be\label{BHIud}
\inf_\Omega \frac{u}{d}\ge
D^{-n/\epsilon}e^{-{C_0(r_\Omega^{-1}+M)D}} \Bigl\| \frac{u}{d}\Bigr\|_{L^\epsilon(\Omega)}.
\ee
Let $r$ be the number given  by the previous lemma.
By Proposition~\ref{geodes}(ii), we may cover $\Omega$ by $N$ balls $B_{r}(x_i)$ in such a way that:
$$
\hbox{$N\le C(n)(D_0/r)^n$,}\qquad
\hbox{$|\Omega\cap B_r(x_i)|\ge c(n)r^n$,}\qquad
\hbox{and either $x_i\in\partial\Omega$ or $d(x_i)\ge 3r/2$}.
$$
Denote $\tilde\Omega=r^{-1}(\Omega-x_i)$, $B_R^+:=B_R\cap \tilde\Omega$, and set
$$
\omega_i=
\begin{cases}
B_2,& \hbox{if $x_i\in \Omega$}\\
\tilde \omega,& \hbox{if $x_i\in\partial\Omega$,}
\end{cases}
$$
where $\tilde \omega$ is the $C^{1,\bar\alpha}$ domain given by Lemma~\ref{LemOptimizedHopf3},
with norm less than~$C(n,\bar\alpha)$, which satisfies
$B_{3/2}^+\subset\tilde \omega\subset B_{2}^+$.
For each $i$, the function $v_i(y):=u(x_i+ry)$ solves
$$\tilde{\mathcal{L}}_iv_i
=g_i(y),
\quad y\in \omega_i$$
where $\tilde{\mathcal{L}}_i$ is defined by \eqref{def-A-resc1}-\eqref{def-A-resc2} with $x_0=x_i$, and $g_i(y)=r^2f(x_i+ry)$.
If $x_i\not\in\partial\Omega$, it follows from Lemma \ref{LemOptimizedHopf2}(i) that
$$\inf_{B_1}v_i\ge C_0\int_{B_1}g_i,
\quad\hbox{ or equivalently}\ \
\inf_{B_r(x_i)}u\ge C_0r^{2-n}\int_{B_r(x_i)} f,\quad\hbox{ hence}$$
\be\label{infvi1}
\inf_{B_r(x_i)} \frac{u}{d}\ge C_0D_0^{-2}r^{2-n}\int_{B_r(x_i)} fd.
\ee
If $x_i\in\partial\Omega$, it follows from Lemma \ref{LemOptimizedHopf2}(ii) that
$$\inf_{B_1^+}\frac{v_i}{d_i}\ge C_0 \int_{B_1^+} g_id_i,
\quad\hbox{where } d_i(y)={\rm dist}(y,\omega_i).$$
We claim that
\be\label{compdistomegai}
{\rm dist}(y,\partial B_{3/2}^+)\ge \min\bigl(1/2,{\rm dist}(y,\partial \tilde\Omega)\bigr),\quad y\in B_1^+.
\ee
Indeed, for given $y\in B_1^+$, denote $p$ the projection of $y$ onto $\partial B_{3/2}^+$.
If $p\in \partial B_{3/2}$, then $|p-y|\ge 1/2$. Otherwise, $p\in \partial\tilde\Omega$
and $B_{|p-y|}(y)\subset B_{3/2}^+\subset\tilde\Omega$, hence $|p-y|={\rm dist}(y,\partial \tilde\Omega)$.
This proves \eqref{compdistomegai}.
Since ${\rm dist}(y,\partial \tilde\Omega)=r^{-1}{\rm dist}(x_i+ry,\partial\Omega)$ and $r\le D_0$, it follows that
$$
\begin{aligned}
d_i(y)
&\ge {\rm dist}(y,\partial B_{3/2}^+)\ge
\min\bigl(1/2,{\rm dist}(y,\partial \tilde\Omega)\bigr)
\ge (2D_0)^{-1}d(x_i+ry),\quad y\in B_1^+.
\end{aligned}$$
Scaling back to $u$, we get
\be\label{infvi2}
\begin{aligned}
\inf_{\Omega\cap B_r(x_i)}\frac{u}{d}
&\ge (2D_0)^{-1}
\inf_{y\in B_1^+} \frac{u(x_i+ry)}{d_i(y)}\ge
C_0D_0^{-1}
\int_{B_1^+} g_i(z) d_i(z)\,dz \\
&\ge C_0 D_0^{-2}r^2
\int_{B_1^+} f(x_i+rz)d(x_i+rz)\,dz=C_0D_0^{-2}r^{2-n}
\int_{\Omega\cap B_r(x_i)} fd.
\end{aligned}
\ee

Now, since
$$
\sum_{i=1}^N \int_{\Omega\cap B_r(x_i)}fd\ge \int_{\Omega}fd,
$$
we may choose $i=i_0$ such that
\be\label{choice}\int_{\Omega\cap B_r(x_i)}fd\ge N^{-1}\int_{\Omega}fd.\ee
Inequalities \eqref{BHIud}, \eqref{infvi1} and \eqref{infvi2} then yield:
$$\begin{aligned}
\inf_\Omega \frac{u}{d}&\ge
D^{-n/\epsilon}e^{-{C_0(r_\Omega^{-1}+M)D}} \Bigl\| \frac{u}{d}\Bigr\|_{L^\epsilon(\Omega)}
\ge
D^{-n/\epsilon}e^{-{C_0(r_\Omega^{-1}+M)D}}
\Bigl\| \frac{u}{d}\Bigr\|_{L^\epsilon(\Omega\cap B_r(x_i))} \\
&\ge
D^{-n/\epsilon}e^{-{C_0(r_\Omega^{-1}+M)D}}
|\Omega\cap B_r(x_i)|^{1/\epsilon}
\inf_{\Omega\cap B_r(x_i)} \frac{u}{d} \\
&\ge e^{-C_0(r_\Omega^{-1}+M)D} (r/D)^{(n/\epsilon)+2-n}D^{-n}
\int_{\Omega\cap B_r(x_i)} fd.
\end{aligned}$$
By using the inequalities  $N^{-1}\ge C(n) (r/D)^n$, $r/D =\delta r_0/D\ge C_0\bigl((r_\Omega^{-1}+M)D\bigr)^{-1}$, as well as
\eqref{relMr0}, \eqref{choice},
we deduce that
$$\inf_\Omega  \frac{u}{d}\ge
e^{-{C_0(r_\Omega^{-1}}+M)D}(r/D)^{(n/\epsilon)+2}D^{-n}
\int_{\Omega}fd
\ge e^{-{C_1(r_\Omega^{-1}+M)D}} D^{-n} \int_{\Omega}fd,$$
which is the desired result.
\end{proof}

\section{Proof of the differential Harnack estimates}
\label{proofloggrad}

We start by giving a short proof of  statements (ii1) and (iii) in  Theorem \ref{thm4gen}.

\begin{proof}[Proof of Theorem~\ref{thm4gen} (iii)]
{Fix $x_0\in\Omega$
 with $d_0=d(x_0)= \mathrm{dist}(x_0,\partial\Omega)$ and set
$$
r= \min\{r_0, d_0\},
$$
where $r_0$ is the number from \eqref{defr0}.
We rescale our equation
$-\ldu=f$
setting $y=(x-x_0)/r$, $u(x)=\tilde u(y)$,
and the coefficients of $\tilde{\ld}$ being defined by \eqref{def-A-resc2}. We obtain
\begin{equation}\label{rescl1}
 -\tilde{\ld}[\tilde u] = r^2f(x_0+ry)=:\tilde f(y)
\end{equation}
in the unit ball, since $r\le d_0$. Since $r\le r_0$, by the choice of $r_0$ we know (see for instance the computations in \eqref{rescaledLq}) that
$$
\|\tilde b_1\|_{C^\alpha(B_1)}\le C_0, \quad \|\tilde b_2\|_{L^q(B_1)}\le C_0,\quad \| \tilde c\|_{L^q(B_1)}\le C_0.
$$
Applying successively the standard $C^1$-to-$C^0$ estimate (\cite[Theorem 8.32]{GT}, \cite[Chapter 5.5]{Mo}), and the  Harnack inequality for \eqref{rescl1} (\cite[Theorem 8.18]{GT} and remark at the end of \cite[Section 8.10]{GT}), we obtain
$$
\begin{aligned}
|\nabla \tilde u(0)|& \le C_0 \sup_{B_{3/4}}\tilde u + C_0\|\tilde f \|_{L^q(B_1)}\\
&\le C_0\tilde u (0) +  C_0r^{2-n/q}\|f\|_{L^q(B_r(x_0)),}
\end{aligned}
$$
which implies that
$$
r|\nabla u (x_0)|\le C_0u(x_0) + C_0r^{2-n/q}\|f\|_{q,r_0,\Omega}.
$$
Hence, multiplying by $d_0/ru(x_0)$ and recalling the definition of $r$ and \eqref{relMr0}, we get
$$\begin{aligned}
\frac{ d_0|\nabla u(x_0)|}{u(x_0)}
&\le C_0 \max\Bigl\{ 1,\frac{d_0}{r_0}\Bigr\} +
C_0 \bigl(\min\{r_0,d_0\}\bigr)^{1-n/q}
 \frac{\|f\|_{q,r_0,\Omega}}{u(x_0)/d(x_0)}\\
&\le C_0 \max \bigl\{1, (r_\Omega^{-1}+M)d_0 \bigr\} +
C_0 \bigl(\min\{(r_\Omega^{-1}+M)^{-1},d_0\}\bigr)^{1-\frac{n}{q}}
 \frac{\|f\|_{q,r_0,\Omega}}{u(x_0)/d_0}.
\end{aligned}$$
Since $x_0\in \Omega$ is arbitrary, this also proves  assertion~(ii1) ($f\equiv 0$).  
If $f\not\equiv 0$, we can next
apply the optimized Morel-Oswald estimate \eqref{conclHopfgen} to the last denominator in this expression, and deduce
$$
{\frac{ d|\nabla u|}{u}}\le C_0 \max\bigl\{1,(r_\Omega^{-1}+M)d\bigr\} +
C_0 \bigl(\min\{(r_\Omega^{-1}+M)^{-1},d\}\bigr)^{1-\frac{n}{q}} e^{C_0(r^{-1}_\Omega+ M)D}D^n\frac{\|f\|_{q,r_0,\Omega}}{\|f\|_{L^1_{d}(\Omega)}},
$$
hence assertion~(iii).
We note that, in view of the exponential factor, we can discard the min in the last inequality without loss of information;
likewise, we cause no loss by replacing $\|f\|_{q,r_0,\Omega}$ with $\|f\|_{L^q(\Omega)}$ -- cf.~Remark~\ref{remnoloss} below.}
\end{proof}
\smallskip

We now turn to the (completely different) proof of the full Theorem~\ref{thm4gen}, based on a contradiction and doubling-rescaling argument
instead of the Harnack inequality.
For simplicity we will give this proof only in the case $\Omega=B_R$, but it can be easily
modified to the case of $C^{1,\bar\alpha}$ domains,
as the reader could verify. This proof is somewhat longer but it allows to show that the  logarithmic gradient bound is independent of $f$ for $n=1$.

\begin{proof}[Proof of  Theorem~\ref{thm4gen} (ii)-(iii).]
Let $\Omega=B_R$. By rescaling $x\to x/R$ and by using Proposition \ref{scaleinvtilde} (ii) we can suppose $R=1$.

We first transform the equation by the change of variable $v:=\log u$, $u=e^v$.
We compute
$$
\begin{aligned}
\mathcal{L}u
&=\nabla\cdot(A(x)e^v\nabla v+b_1(x)e^v)+e^vb_2(x)\cdot\nabla v+c(x)e^v \\
&=e^v\Bigl\{\nabla\cdot(A(x)\nabla v)+\nabla\cdot b_1(x)+\nabla v\cdot A(x)\nabla v
+b_1(x)\cdot\nabla v+b_2(x)\cdot\nabla v+c(x)\Bigr\} \\
\end{aligned}
$$
hence, setting $b=b_1+b_2$,
\be\label{eqforvk}
\mathcal{L}_1v
:=\nabla\cdot(A(x)\nabla v)+\nabla v\cdot A(x)\nabla v+b(x)\cdot\nabla v+\nabla\cdot b_1(x)+c(x)
=-e^{-v}f(x).
\ee
In terms of $v$, the sought-for estimate is equivalent to
$$
d|\nabla v| \le
\begin{cases}
C_0\max(1,Md),& \hbox{if $f\equiv 0$ or $n=1$}\\
\noalign{\vskip 1mm}
C_0\left(\max(1,Md)+d^{1-n/q}e^{C_1M}\frac{\|f\|_{q,r_0,B_1}}{\|f\|_{L^1_d(B_1)}}\right),
& \hbox{otherwise}
\end{cases}
$$
where $M=M(\ld,B_1)$, $r_0=r_0(\ld,B_1)$ (cf.~\eqref{defM}) and, in the second case, $C_1$ is the constant from the Morel-Oswald estimate
in Theorem~\ref{thm4gen} (i), which we already proved.

Assume for contradiction that there exist  sequences of  operators $\ld_k$ in the form \eqref{eqforvk}, with coefficients
$A_k, b_k, b_{1,k}, c_k$
satisfing \eqref{hyp1}-\eqref{hyp2}, $f_k\in L^q(B_1)$, solutions $v_k$ of \eqref{eqforvk} and points $y_k\in B_1$ for which
$$L_k(y_k)\ge 2k (d(y_k))^{-1},$$
with $M_k=M(\ld_k,B_1)$, $r_k=r_0(\ld_k,B_1)$ and
\be\label{defMk}
L_k(x):=
\begin{cases}
\ds\frac{|\nabla v_k(x)|}{\max(1,M_kd(x))},& \hbox{if $f_k\equiv 0$ or $n=1$}\\
\noalign{\vskip 1mm}
\ds\frac{|\nabla v_k(x)|}{{\max(1,M_kd(x))+e^{C_1M_k}d^{1-n/q}(x)\frac{\|f_k\|_{q,r_k,B_1}}{\|f_k\|_{L^1_d(B_1)}}}
},& \hbox{otherwise}.
\end{cases}
\ee
By the doubling lemma (cf.~\cite[Lemma 5.1]{PQS}), there are points $x_k\in B_1$ such that
\be\label{largenessMk}
L_k(x_k)\ge L_k(y_k), \qquad |\nabla v_k(x_k)| 
\ge L_k(x_k)\ge 2k (d(x_k))^{-1} \ge 2k,\quad\mbox{and}
\ee
\be\label{doublingMk}
L_k(z)\le 2 L_k(x_k)\qquad \mbox{if }\: |z-x_k|< k (L_k(x_k))^{-1}.
\ee
Set
\be\label{defrk}
\rho_k= |\nabla v_k(x_k)|^{-1} \qquad (\rho_k\to 0 \mbox{ as } k\to \infty).
\ee
Observe for later purposes that
\be\label{comprk}
\rho_k
\le (2k)^{-1}\min\{M_k^{-1},d(x_k)\}
\ee
owing to \eqref{defMk}, \eqref{largenessMk}, \eqref{defrk}.
We now rescale
\be\label{rescalingvk}
x=\rho_ky + x_k, \quad w_k(y) =  v_k(\rho_ky+x_k)-v_k(x_k),
\ee
i.e., $v_k(x) = w_k(\rho_k^{-1}(x-x_k))+v_k(x_k)$.
Note that, while in \eqref{eqforvk} we have an equation in $v$, we do the rescaling in such a way that only the gradients $\nabla w_k$ of the rescaled $v_k$ (rather than the $w_k$ themselves) stay bounded.
We obtain from \eqref{eqforvk} (omitting the variable $y$):
\be\label{eqforwk}
\nabla\cdot(\tilde A_k\nabla w_k)+\nabla w_k\cdot \tilde A_k\nabla w_k+\tilde b_k\cdot\nabla w_k+
\nabla\cdot \tilde b_{1,k}+\tilde c_k=-g_k,
\ee
where
$$
\tilde{A}_k(y) := A_k(\rho_ky + x_k),\quad
\tilde{b}_k(y) := \rho_k b_k(\rho_ky + x_k),\quad \tilde{b}_{1,k}(y) := \rho_k b_{1,k}(\rho_ky + x_k),$$
$$\tilde{c}_k(y) := \rho_k^2 c_k(\rho_ky + x_k),\quad \tilde{f}_k(y) :=  \rho_k^2 f(\rho_ky + x_k),\
\quad g_k(y):=e^{-w_k(y)}e^{-v_k(x_k)}\tilde f_k(y).$$
For $|y|\le k$, by \eqref{comprk}, we have
$d(x_k+\rho_ky)\le d(x_k)+\rho_k|y|\le  \frac32 d(x_k)$. This combined wih
\eqref{defMk}, \eqref{defrk} and \eqref{rescalingvk} implies
$$
|\nabla w_k(y)| = \rho_k |\nabla v_k(x_k+\rho_ky)|= \frac{|\nabla v_k(x_k+\rho_ky)|}{|\nabla v_k(x_k)|}
\,\le \frac32\frac{L_k(x_k+\rho_ky)}{L_k(x_k)},
$$
hence, owing to \eqref{doublingMk},
\be\label{boundDwk}
|\nabla w_k(y)| \le 3\quad\hbox{for $|y|<k$},\quad\mbox{and}
\ee
\be\label{boundDwk0}
|\nabla w_k(0)| = \rho_k |\nabla v_k(x_k)| = 1,\qquad w_k(0)=0.
\ee
As a consequence of the second part of \eqref{boundDwk0} and of \eqref{boundDwk}, we have
\be\label{boundwk}
|w_k(y)|\le 3|y|\quad\hbox{for $|y|<k$}.
\ee

 {Fix any $R_0>1$. Note that \eqref{comprk} and \eqref{relMr0} guarantee that for $k$ large (depending on $R_0$),
we have
{$\rho_k^{-1}\ge 2k\max\{M_k,d(x_k)^{-1}\}>
R_0(r_{B_1}^{-1}+M_k)=R_0r_k^{-1}$, hence} $R_0\rho_k<r_k$. We then obtain, as $k\to\infty$:}
\be\label{cvbk}
\|\tilde{b}_k\|_{L^q(B_{R_0})}\le
\rho_k^{1-n/q} \|{b}_k\|_{L^q({B_1 \cap B_{\rho_kR_0}(x_k))}}\le
\rho_k^{1-n/q} \|{b}_k\|_{q,r_k,B_1}\le (\rho_kM_k)^{1-n/q} \to 0,
\ee
\be\label{cvb11k}
\|\tilde{b}_{1,k}\|_{L^\infty(B_{R_0})}=\rho_k\|b_{1,k}\|_{L^\infty(B_{R_0})} \le
(2k)^{-1} M_k^{-1} \|b_{1,k}\|_{L^\infty(B_{R_0})} \le (2k)^{-1}
\to 0,
\ee
and similarly
\be\label{cvck}
\|\tilde{c}_k\|_{L^q(B_{R_0})}\to 0,\quad [\tilde{A}_k]_{\alpha, B_{R_0}}\to 0,\quad [\tilde{b}_{1,k}]_{\alpha, B_{R_0}}\to 0.
\ee
Therefore, passing to a subsequence, we may assume that $x_k\to x_\infty\in \overline B_1$
and  $A_k(x_k)\to A_\infty$, where $A_\infty$ is a constant
positive definite matrix.
By the compact embedding $C^\alpha\hookrightarrow C^0$,
$\tilde A_k(y)\to A_\infty$ uniformly for $y$ in any compact set.

\medskip

{\bf Case 1: $n\ge 2$.}
If $f_k\not\equiv 0$, we also need to estimate the right hand side $g_k$.
As above, we first write
\be\label{cvfk}
\|\tilde{f}_k\|_{L^q(B_{R_0})}\le
\rho_k^{2-n/q} \|{f}_k\|_{L^q({B_1\cap B_{\rho_kR_0}(x_k))}}
\le  \rho_k^{2-n/q} \|{f}_k\|_{q,r_k,B_1}.
\ee
We then use Morel-Oswald estimate
in Theorem~\ref{thm4gen} (i) and the second part of \eqref{largenessMk} to write
$$u_k(x_k)\ge e^{-C_1(1+M_k)}\|f_k\|_{L^1_d(B_1)}\,d(x_k)\ge e^{-C_1(1+M_k)}\|f_k\|_{L^1_d(B_1)} \ 2k L_k^{-1}(x_k)$$
hence,
$$\begin{aligned}
e^{-v_k(x_k)}
&=u_k^{-1}(x_k)\le C_0e^{C_1M_k}\|f_k\|_{L^1_d(B_1)}^{-1} (2k)^{-1} L_k(x_k). \\
\end{aligned}$$
Therefore, using
 \eqref{boundwk}, \eqref{cvfk}, and the definition \eqref{defMk} of $L_k(x_k)$ we get, for any $R_0>1$ and $k>R_0$,
$$
\begin{aligned}
\|g_k\|_{L^q(B_{R_0})}
&= e^{-v_k(x_k)} \|e^{-w_k}\tilde f_k\|_{L^q(B_{R_0})}
\le  C_0e^{C_1M_k}\|f_k\|_{L^1_d(B_1)}^{-1} L_k(x_k) e^{3R_0}\rho_k^{2-n/q} \|{f}_k\|_{q,r_k,B_1}\\
&\le {(2k)^{-2+n/q} e^{3R_0}L_k(x_k)e^{C_1M_k} \frac{ \|{f}_k\|_{q,r_k,B_1}}{\|f_k\|_{L^1_d(B_1)}}  d^{1-n/q}(x_k) |\nabla v_k(x_k)|^{-1}} \\ 
&\le (2k)^{-2+n/q} e^{3R_0} \to 0.
\end{aligned}
$$
By the last inequality, together with  \eqref{boundDwk} and \eqref{cvbk}-\eqref{cvck}, the equation \eqref{eqforwk} is in the form
$$
\mathrm{div} (\tilde A_k\nabla w_k) = -\mathrm{div}(\tilde{b}_{1,k}) + h_k
$$
where $\tilde{b}_{1,k}$ is bounded in $C^\alpha(B_{R_0})$ while $h_k$ is a function bounded in $L^q(B_{R_0})$, as $k\to\infty$. It follows from the $C^{1,\nu}$ interior
estimate in \cite[Theorem 8.32]{GT}, \cite[Chapter 5.5]{Mo},
that, for some $\nu\in(0,1)$, $w_k$ is bounded in $C^{1,\nu}_{loc}(\R^n)$,
and then, up to a subsequence,
$w_k$ converges in $C^1_{loc}(\R^n)$ to a weak solution $w\in C^{1,\nu}_{loc}(\R^n)$ of the limiting equation of \eqref{eqforwk}, namely
$$\nabla\cdot(A_\infty\nabla w)+\nabla w\cdot A_\infty\nabla w=0,\quad
x\in\R^n$$
with $|\nabla w(0)|=1$.
Thus the function $U=e^w$ is  a positive solution of
$\nabla\cdot(A_\infty\nabla U)=0$ in $\R^n$ (and $U$ is in fact a classical solution by elliptic regularity). Since $A_\infty$ is a constant matrix we have $\nabla\cdot(A_\infty\nabla U) = \mathrm{tr}(A_\infty D^2U)$, so after a linear change of the independent variable we obtain a positive harmonic function in $\rn$,
which must be constant, by the mean value property.
Since $|\nabla U(0)|\ne 0$, this is a contradiction. 

\medskip

{\bf Case 2: $n=1$.}
Pick any nonnegative, compactly supported $\varphi\in H^1(\R)$.
Testing \eqref{eqforwk} for $n=1$ with $\varphi$, we get
\be\label{eqforwk2}
\int\tilde A_kw'_k\varphi'-\int\tilde A_k(w_k')^2\varphi-\int\tilde b_k w'_k \varphi+\int \tilde b_{1,k}\varphi'-
\int\tilde c_k\varphi=\int g_k\varphi \ge 0.
\ee
In particular, taking $R_0>1$ and assuming
$supp(\varphi)\subset (-2R_0,2R_0)$, $0\le\varphi\le 1$,
$|\varphi'|\le1$ and $\varphi=1$ on $[-R_0,R_0]$, we obtain by \eqref{boundDwk}, \eqref{cvbk}-\eqref{cvck},
$$\int_{|y|<R_0} g_k\le 9 \Lambda R_0+3\|\tilde b_k\|_{L^1(B_{2R_0})}+4R_0\|b_{1,k}\|_\infty+\|\tilde c_k\|_{L^1(B_{2R_0})}
\le C(R_0).$$
Therefore, $g_k$ is bounded in $L^1_{loc}(\R)$.
{In view of \eqref{eqforwk}-\eqref{boundDwk} and \eqref{cvbk}-\eqref{cvck}, it follows that
$(\tilde A_kw_k'+\tilde b_{1,k})'$  is bounded in $L^1_{loc}(\R)$ hence, using again \eqref{boundDwk}-\eqref{cvb11k},
$\tilde A_kw_k'+\tilde b_{1,k}$  is bounded in $W^{1,1}_{loc}(\R)$.
By this along with \eqref{boundDwk} and \eqref{boundwk} (which say $w_k$ is bounded in $W^{1,\infty}_{loc}(\R)$), up to extracting a subsequence, there exist functions
$w\in W^{1,\infty}_{loc}(\R)$ and $z\in L^\infty_{loc}(\R)$ such that $\tilde w_k\to w$ in $L^\infty_{loc}(\R)$
and $\tilde A_kw_k'+\tilde b_{1,k}\to z$ in
$L^m_{loc}(\R)$ for all finite $m$ and almost everywhere. Up to extracting a further subsequence, we may also assume
that $\tilde A_k(y)\to a_\infty$ in $L^\infty_{loc}(\R)$ for some number $a_\infty>0$. Therefore, recalling \eqref{cvb11k},
$w_k'$ converges in $L^m_{loc}(\R)$, hence
\be\label{convwkw}
w_k\to w \quad\hbox{ in $W^{1,m}_{loc}(\R)$ for all finite $m$.}
\ee

We next claim that $w$ is nonconstant. Namely we will show that
\be\label{nonvanishw}
\int_{-1}^1 |w'(x)|dx>0.
\ee
To this end, to overcome the lack of $C^1$ convergence,
we need to ``thicken'' the normalization condition \eqref{boundDwk0}.
By extracting a further subsequence, we may assume that $w_k'(0)=-1$ for all $k$
or that $w_k'(0)=1$  for all $k$.
In the first case, by \eqref{eqforwk}-\eqref{boundDwk} and \eqref{cvbk}-\eqref{cvck}, there exists a sequence $\eps_k\to 0^+$ such that, for all $x\in [0,1]$,
$$\bigl[\tilde A_kw_k'\bigr]_0^x\le \int_0^x|\tilde b_kw_k'+\tilde c_k|ds+\bigl|\bigl[\tilde b_{1,k}\bigr]_0^x\bigr|\le \eps_k.$$
Therefore, there exists $\eta>0$ such that, for all $k$ sufficiently large,
$$w_k'(x)\le \frac{-\tilde A_k(0)+\eps_k}{\tilde A_k(x)}\le -\eta,\ \hbox{ for all $x\in[0,1]$.}$$
In the second case, we similarly get $w_k'(x)\ge \eta$ for all $x\in[-1,0]$.
In either case, by the convergence \eqref{convwkw}, we  have
$\int_{-1}^1 |w'(x)|dx\ge \eta$, which proves the claim.

Now, by virtue of \eqref{convwkw}, we may pass to the limit
in \eqref{eqforwk2} with the help of \eqref{cvbk}--\eqref{cvck}, to deduce that, for any nonnegative, compactly supported $\varphi\in H^1(\R)$, we have
$\int a_\infty w'\varphi'-\int a_\infty(w')^2\varphi \ge 0$, hence
$\int w'\varphi'- (w')^2\varphi \ge 0$.
For any nonnegative $\psi\in C^\infty_0(\R)$, we may then take $\varphi:=e^w\psi$, which gives
$$\int (e^w)'\psi'=\int w'(e^w(w'\psi+\psi'))- (w')^2e^w\psi =\int w'\varphi'- (w')^2\varphi\ge 0.$$
In other words, $(e^w)''\le 0$ in the distributional sense.
Fix a mollifying sequence $(\rho_j)$. For all $j$, it follows in particular that
$(e^w\ast\rho_j)''\le 0$ in $\R$. The smooth function $e^w\ast\rho_j\ge 0$ is thus concave on $\R$ and
must therefore be constant, hence $(e^w\ast\rho_j)'=0$.
Recalling that $e^w\in W^{1,\infty}_{loc}(\R)$, we have $0=(e^w\ast\rho_j)'=(e^w)'\ast\rho_j\to (e^w)'$
in $L^1_{loc}(\R)$. Consequently $w^\prime e^w=(e^w)'=0$ a.e., hence $w'=0$ a.e.: a contradiction with \eqref{nonvanishw}.}
\end{proof}

\section{Sharpness of the Morel-Oswald estimate}
\label{sec-proofoptim}

\begin{proof}[Proof of Proposition~\ref{Prop-Optim-Hopf0}]
We can assume  that $R=1$. Indeed, applying the transformation $\tilde u(x)=u(y)=u(Rx)$, by the scaling property \eqref{scaleinvtilde5} in Proposition \ref{scaleinvtilde}, the fact that $Rd_1(x)=d_R(Rx)$ (here $d_R(y)=\mathrm{dist}(y, \partial B_R)$), and
$\int_{B_1}\tilde fd_1=
R^{1-n}\int_{B_R}fd_R$, we  have
$$ e^{-C_0(1+\tilde M)}\left(\int_{B_1}\tilde fd_1\right)d_1(x)=e^{-C_0(1+MR)}R^{-n}\left(\int_{B_R}fd_R\right)d_R(y).$$

Fix any $\lambda>0$ and let $\phi=\frac12(1-|x|^2)$.
We look for a (radial) solution of the form
$$u=e^{\lambda\phi}-1+K\phi,$$
with $K\ge 0$ to be determined.
We note  that $u>0$ in $B_1$ and $u=0$ on~$\partial B_1$.
By direct computation we have
$$\nabla\phi=-x,\quad \Delta\phi=-n,$$
$$\nabla e^{\lambda\phi}=\lambda e^{\lambda\phi}\nabla\phi=-\lambda e^{\lambda\phi}x,\qquad
\Delta e^{\lambda\phi}=e^{\lambda\phi}\bigl(\lambda \Delta\phi+\lambda^2 |\nabla\phi|^2\bigr)
=e^{\lambda\phi}\bigl(-\lambda n+\lambda^2 |x|^2\bigr).$$
For any smooth functions $b, c$ such that $b\cdot x\le 0$ and $c\ge 0$,
choosing $K=n^{-1}\|c\|_\infty$, it follows that
$$\begin{aligned}
f:=-\Delta u+b\cdot\nabla u+cu
&=e^{\lambda\phi}\bigl(\lambda n-\lambda^2|x|^2-\lambda b\cdot x+c\bigr)-c+K(n-b\cdot x+c\phi) \\
&\ge e^{\lambda\phi}\bigl(\lambda n-\lambda^2|x|^2-\lambda b\cdot x+c\bigr).
\end{aligned}$$
We now choose
$$
c=c_\lambda=\lambda^2|x|^2\;\mbox{ and }\; b=0,\qquad\mbox{or}\qquad
b=b_\lambda=-\lambda x\;\mbox{ and }\; c=0.
$$
Then
\be \label{optimHopf3}
\lambda=\|c\|_\infty^{1/2} \quad\hbox{(resp., $\|b\|_\infty$),}
\ee
 and in either case we have $f\ge\lambda n e^{\lambda\phi}$, hence
\be \label{optimHopf4}
\int_{B_{1}} fd
\ge \lambda n\ds\int_{0<|x|<1/2} e^{\lambda(1-|x|^2)/2} (1-x)dx\ge C(n)\lambda e^{3\lambda/8}.
\ee
On the other hand, on the boundary $\partial B_1$ we have $|u_\nu|=-x\cdot \nabla u = \lambda+K=\lambda+n^{-1}\lambda^2$ (resp.~$|u_\nu|=\lambda$ if $b=-\lambda x$, $c=0$),
hence $|u_\nu|\le \lambda(1+\lambda)$ in either case.
Consequently,
we deduce from \eqref{optimHopf4}  that
$$|u_\nu|\,\Bigl(\int_{B_1} fd\Bigr)^{-1}\le C(n)\lambda(1+\lambda)e^{-3\lambda/8}
\le C(n)e^{-\lambda/4}\le C(n)e^{-C_1(n)M},$$
where we used \eqref{optimHopf3} and Proposition~\ref{lemr0ul}(i) to get the last inequality.
Finally, by Proposition~\ref{lemr0ul}(ii), the function
$M(\lambda):=\|c_\lambda\|^{\gamma_q}_{q,r_\lambda,B_1}$ {with $r_\lambda=r_0(\Delta-c_\lambda, B_1)$}
is continuous on $(0,\infty)$ and $\lim_{\lambda\to \infty}M(\lambda)=\infty$,
whereas $\lim_{\lambda\to 0}M(\lambda)=0$ owing to \eqref{culinfty}.
Therefore the range of $M(\lambda)$ is the whole half-line $(0,\infty)$ (and similarly for $b_\lambda$),
which completes the proof.
\end{proof}

\begin{rem} \label{Hopf-optL1d}
As mentioned in the introduction, the $L^1_d$ norm in the right-hand side of the quantitative Morel-Oswald inequality \eqref{conclHopf} cannot be replaced by the stronger norms $\|\cdot\|_{L^p_d}$ or $\|\cdot\|_{L^1_{d^\epsilon}}$
for $p>1$ or $\epsilon<1$.
Indeed, consider for instance the case of the Laplacian, and assume for contradiction that the estimate
$u(x)\ge C_1 \|f\|d(x), \ x\in\Omega,$ is true for one such norm of $f$ with some constant $C_1>0$,
for each $u$ and $f$ such that $-\Delta u \ge f\ge0$ in $B_1$.
Denoting by $G(x,y)$ the Green function of the Dirichlet Laplacian in $B_1$
we have, by classical estimates (see e.g.~\cite{Zh} and the references therein):
$$G(x,y)\le Cd(x)d(y)\quad\hbox{for all $x,y\in B_1$ such that $|x-y|\ge 1/4$}.$$
Consequently, for any nonnegative $f\in C(\overline B_1)$ with support in $\overline B_1\setminus B_{1/2}$,
$$C_1\|f\| \le  u(0)=\int_\Omega G(0,y)f(y)dy\le C\int_\Omega f(y)d(y)dy= C\|f\|_{L^1_d},$$
a contradiction.\end{rem}

\begin{rem} \label{Hopf-optC1}
Furthermore, the Morel-Oswald inequality fails if we take $\alpha=0$ in the assumptions on the $C^\alpha$-regularity of the leading order coefficients of $\ld$ or  the $C^{1,\alpha}$-regularity of $\partial\Omega$. Indeed, say $0\in\partial\Omega$, $\partial \Omega$ is $C^\infty$-smooth except at the origin,  in a neighbourhood of $0$ we have $\partial\Omega = \{x=(x^\prime,x_n)\::\:  \psi(|x^\prime|)<x_n<r_0\}$ for some $r_0>0$ and a nonnegative  function $\psi\in C^1(0,r_0)$ such that $\int_0 t^{-2}\psi(t)\,dt=\infty$ (for instance $\psi(t) = t(|\log|t||^{-1})$). If $f\in C^\infty(\overline{\Omega})$ is such that $0\le f\le1$ in $\Omega$, $f=0$ in $\Omega\cap\{x_n<r_0/3\}$ and $f=1$ in $\Omega\cap\{x_n>2r_0/3\}$, then the Hopf-Oleinik lemma (and a fortiori the Morel-Oswald estimate) fails for the positive solution of the Dirichlet problem $-\Delta u = f$ in $\Omega$, $u=0$ on $\partial \Omega$ - see
 \cite[Theorem 1.3]{Saf}. Next, effecting the change of variables $y^\prime=x^\prime$, $y_n=x_n-\psi(|x^\prime|)$, $v(y)=u(x)$, $\tilde f(y)=f(x)$, we obtain a positive solution to an equation $\mathrm{div}(A(y)\nabla v) = \tilde f (y)$ in a smooth domain in which $A\in C^0$, for which the Hopf-Oleinik lemma fails.
\end{rem}

\section{Sharpness of the global and differential Harnack  estimates}
\label{sec-proofoptim2}

We first show that the  exponential nature of the constant in front of the right-hand side $f$ in the Harnack inequality (Theorem \ref{BHIoptim}) cannot be improved.

\begin{prop}\label{Optimality_Harnack2}
Let $n\ge 1$, $n<q\le\infty$, and $R>0$.
There exist a constant $\lambda_0=\lambda_0(n)>0$,  sequences $b_j, c_j,f_j\in C^\infty(\overline{B_R})$,
and a classical solution $u_j>0$ of
\be \label{optimLinfty1d}
-\ld_ju_j:=-\Delta u_j+c_ju_j=f_j\; \mbox{ in } B_R,\qquad
u_j=0 \; \mbox{ on }\partial B_R,
\ee
respectively of
\be \label{optimLinfty1db}
-\ld_ju_j:=-\Delta u_j + b_j\cdot\nabla u_j - \lambda_0R^{-2} u_j=f_j \; \mbox{ in } B_R,\qquad
u_j=0 \; \mbox{ on } \partial B_R,
\ee
such that, for some $C=C(n)>0$, 
\be \label{optimLinfty24}
\inf_{B_R}(u_j/d)=0\qquad\mbox{and}\qquad\sup_{B_R}(u_j/d)\ge CR^{1-n/q}\exp\left(CM_jR\right) \| f_j\|_{q,r_j,B_R},
\ee
where $M_j=\|c_j\|^{\gamma_q}_{q,r_j,B_R}\to\infty$, resp.~$M_j=\|b_j\|^{\beta_q}_{q,r_j,B_R}\to\infty$,
with $r_j=r_0(\ld_j,B_R)$.

\end{prop}

We will  use the following simple lemma.

\begin{lem} \label{Mitoinfty}
Let $R,\eta>0$ and let $\ld_i$ be a sequence of operators of the form \eqref{defdiv}.
Assume that
$\displaystyle\lim_{i\to\infty}\,\Bigl\{\sup_{x\in\overline B_R} \inf_{B_\eta(x)\cap B_R} |c_i|\Bigr\}=\infty$
(or the same for $|b_i|$ instead of $|c_i|$).
Then $\displaystyle\lim_{i\to\infty} M(\ld_i,B_R)=\infty$.
\end{lem}

\begin{proof}
Set $M_i=M(\ld_i,B_R)$ and $r_i:=r_0(\ld_i,B_R)$.
If the conclusion fails we may assume that $\sup_i M_i<\infty$.
By \eqref{relMr0} it follows that $\bar r:=\inf_i r_i>0$.
Therefore, setting $\rho=\min(\bar r,\eta)$, we deduce that
$$M_i^{q/\gamma_q}\ge \|c_i\|^q_{q,r_i,B_R}\ge \sup_{x\in\overline B_R} \int_{B_\rho(x)\cap B_R} |c_i|^q\,dx
\ge c(n)\rho^n \sup_{x\in\overline B_R} \inf_{B_\eta(x)\cap B_R} |c_i|^q\to\infty,$$
a contradiction. The same argument applies in the case of $b_i$.
\end{proof}

\begin{proof}

It suffices to assume $R=1$.
Indeed, as in the proof of Proposition~\ref{Prop-Optim-Hopf0} above, if $\tilde u_j(x)=u_j(y)=u_j(Rx)$, by the properties \eqref{scaleinvtilde4}-\eqref{scaleinvtilde5}
and  $Rd_1(x)=d_R(Rx)$, we  have
$$R\sup_{y\in B_R}\frac{u_j(y)}{d_R(y)}=\sup_{x\in B_1}\frac{\tilde u_j(x)}{d_1(x)}\ge C\exp(C\tilde M_j) \|\tilde f_j\|_{q,\tilde r_j,B_1}
=C\exp(CM_jR) R^{2-n/q}\|f_j\|_{q,r_j,B_R}.$$

Let $\lambda_0>0$ be the first eigenvalue of $-\Delta$ in $B_1$ with Dirichlet boundary conditions
and $\varphi(x)=\tilde\varphi(r)\in C^\infty(\overline B_1)$ be the corresponding (radial) eigenfunction with $\varphi(0)=1$.

\smallskip

{\bf Case 1:} Problem \eqref{optimLinfty1d}.
Let $h=\varphi^2$, $u_j=e^{j h}-1$, $c_j=-j\Delta h-j^2|\nabla h|^2$ and $f_j=c_j$. Then
$$
-\Delta u_j=
-\bigl[j\Delta h+j^2|\nabla h|^2\bigr]e^{j h}=c_ju_j+f_j.
$$
Since $u_j=e^{j h}-1\sim j \varphi^2=O(d^2)$ as $x\to\partial B_1$, we have $\inf_{B_1}(u_j/d)=0$.
On the other hand, $\sup_{B_1}(u_j/d)\ge u_j(0)=e^j-1$.
Since
{$M_j\le C \|c_j\|_\infty^{1/2} \le Cj$ owing to \eqref{culinfty}, we get
$$\bigl(\|f_j\|_{L^q(B_1)}\bigr)^{-1}\sup_{B_1}(u_j/d)=\bigl(\|c_j\|_{L^q(B_1)}\bigr)^{-1}\sup_{B_1}(u_j/d)\ge j^{-2}(e^j-1)\ge Ce^{j/2}\ge Ce^{CM_j}.$$
Since $M_j\to\infty$ by Lemma~\ref{Mitoinfty}}, the conclusion
follows.

\smallskip

{\bf Case 2:} Problem \eqref{optimLinfty1db}. This case is more delicate.
Let $h(x)=\tilde h(r)$ and $\theta(x)=\tilde\theta(r)$ be radial nonincreasing $C^\infty$ functions, such that
$$\hbox{$\tilde\theta(r)=1$ for $r\le 1/2$, \ \ $\tilde\theta^\prime(r)<0$ for $1/2<r<1$,\ \ $\tilde\theta(1)=\tilde\theta^\prime(1)=0$,}$$
$$\hbox{$\tilde h(r)=1$ for $r\le 1/3$, \ \ $\tilde h(r)=0$ for $2/3\le r\le 1$, \ \ $h(1/2)=1/2$}, \quad\mbox{and}$$
\be \label{optimLogg-nabla0}
\hbox{$\tilde h'(r)<-1/10$ for $3/8\le r\le 5/8$}.
\ee
For any given integer $j\ge 1$, we set
$$u_j:=\theta\varphi e^{j h},$$
which in particular satisfies $u_j=0$ on $\partial B_1$ and $u_j>0$ in $B_1$.
We compute
\be \label{optimLogg-nabla1}
\nabla u_j=\bigl[\nabla(\theta\varphi)+j\theta\varphi\nabla h\bigr]e^{j h},
\ee
$$\Delta u_j
=\bigl[\Delta(\theta\varphi)+2j\nabla(\theta\varphi)\cdot\nabla h+\theta\varphi(j\Delta h+j^2|\nabla h|^2)\bigr]e^{j h}.$$
Setting $f_j=-\bigl(2\nabla\theta\cdot\nabla\varphi+\varphi\Delta\theta\bigr)e^{j h}$ and using $\Delta(\theta\varphi)=-\lambda_0\theta\varphi+2\nabla\theta\cdot\nabla\varphi+\varphi\Delta\theta$,
we get
\be \label{optimLogg-nabla11}
\Delta u_j+\lambda_0 u_j+f_j
=\bigl[2j\nabla(\theta\varphi)\cdot\nabla h+\theta\varphi(j\Delta h+j^2|\nabla h|^2)\bigr]e^{j h(x)}.
\ee
Also we observe that
\be \label{optimLogg-nabla11a}
\inf_{B_1}\frac{u_j}{d}\le\lim_{|x|\to 1}\frac{u_j}{d}=|(\tilde\theta\tilde\varphi)'(1)|=0
\ee
and
$$\sup_{B_1}\frac{u_j}{d}\ge \frac{u_j}{d}(0)=e^j,
\qquad |f_j|\le C\chi_{\{1/2\le |x|\le 1\}}e^{j h(x)}\le Ce^{j/2},$$
hence
\be \label{optimLogg-nabla11b}
\sup_{B_1}\frac{u_j}{d}\ge Ce^{j/2}\|f_j\|_\infty.
\ee
Now, since $\varphi$ is radially strictly decreasing and  $\theta$, $h$ are also radially decreasing, we have
\be\label{gradsy}
|\nabla(\theta\varphi)+j\theta\varphi\nabla h|
\ge \max\{|\nabla(\theta\varphi)|, j|\theta\varphi\nabla h|\}\ge \theta|\nabla\varphi|>0, \quad \mbox { for }0<|x|<1.
\ee
We may then set
\be \label{Deltaujbj21}
b_j:=\frac{\nabla(\theta\varphi)+j\theta\varphi\nabla h}{|\nabla(\theta\varphi)+j\theta\varphi\nabla h|^2}
\bigl[2j\nabla(\theta\varphi)\cdot\nabla h+\theta\varphi(j\Delta h+j^2|\nabla h|^2)\bigr],\quad x\in B_1\setminus\{0\}.
\ee
Since $h$ is constant, hence $b_j=0$, in $B_{1/3}$ and in $\overline B_1\setminus B_{2/3}$, the function $b_j$ extends to a function $b_j\in C^\infty(\overline B_1)$,
and \eqref{optimLogg-nabla1}, \eqref{optimLogg-nabla11} and \eqref{Deltaujbj21} imply
$$-\Delta u_j-\lambda_0 u_j+b_j\cdot\nabla u_j=f_j.$$
Moreover, since $h=$const in $B_{1/3}$ and $B_1\setminus B_{2/3}$, by \eqref{gradsy} we obtain
$$\begin{aligned}|b_j|
&=\frac{\bigl|2j\nabla(\theta\varphi)\cdot\nabla h+\theta\varphi(j\Delta h+j^2|\nabla h|^2)\bigr|}{|\nabla(\theta\varphi)+j\theta\varphi\nabla h|}
\le j\frac{\bigl|2\nabla(\theta\varphi)\cdot\nabla h+\theta\varphi\Delta h\bigr|}{|\nabla(\theta\varphi)|}
+\frac{j^2\theta\varphi|\nabla h|^2}{j|\theta\varphi\nabla h|}\\
&\le j\frac{\bigl|2\nabla(\theta\varphi)\cdot\nabla h+\varphi\Delta h\bigr|}{|\theta\nabla \varphi|}+j|\nabla h|\le j\frac{C}{\inf_{1/3\le|x|\le 2/3}\theta|\nabla \varphi|}+Cj \le Cj,
\end{aligned}$$
hence
$M_j\le C\|b_j\|_\infty \le Cj$ owing to \eqref{culinfty}.
Since also by~\eqref{optimLogg-nabla0}, for large $j$
$$
|b_j(x)|\ge \frac{Cj^2|\nabla h|^2- Cj}{C(1+j)}\ge Cj\quad\mbox{ for } \;3/8\le |x|\le 5/8,$$
 we have $M_j\to\infty$ by Lemma~\ref{Mitoinfty}.
This along with \eqref{optimLogg-nabla11a}-\eqref{optimLogg-nabla11b}
yields the conclusion.
\end{proof}

We next turn to the proof of Proposition~\ref{Optimality_Harnack}.

\begin{proof}[Proof of Proposition~\ref{Optimality_Harnack}]
Again it suffices to consider the case $R=1$.
Let $\lambda_0>0$ be the first eigenvalue of $-\Delta$ in $B_1$ with Dirichlet boundary conditions
and $\varphi(x)=\tilde\varphi(r)\in C^\infty(\overline B_1)$ be the corresponding (radial) eigenfunction with $\varphi(0)=1$.
Let $h\ge 0$ be a radial nonincreasing $C^\infty$ function to be fixed below, such that
$h(0)=1$, $h(x)=0$ for $|x|=1$ and
$|\nabla h|\ge c_0>0$ in a neighborhood of $x=1/2$.
For any given integer $j\ge 1$, we set
$$u_j(x):=\varphi(x)e^{j h(x)},$$
which in particular satisfies $u_j=0$ on $\partial B_1$ and $u_j>0$ in $B_1$.
We compute
\be \label{optimLogg-nabla}
\nabla u_j=\bigl[\nabla\varphi+j\varphi\nabla h\bigr]e^{j h(x)},
\ee
$$\begin{aligned}
\Delta u_j
&=\bigl[\Delta\varphi+2j\nabla\varphi\cdot\nabla h+\varphi(j\Delta h+j^2|\nabla h|^2)\bigr]e^{j h(x)}\\
&=\bigl[2j\nabla\varphi\cdot\nabla h+\varphi(j\Delta h+j^2|\nabla h|^2-\lambda_0)\bigr]e^{j h(x)}.
\end{aligned}$$
On the other hand, since $\varphi, h$ are radially decreasing, we have
$$\frac{|\nabla u_j|}{u_j}=|j\nabla h+\varphi^{-1}\nabla\varphi|\ge j |\nabla h|,$$
hence there exists a constant $C_0>0$ such that
\be \label{optimLogg}
\frac{|\nabla u_j(x)|}{u_j(x)}\ge C_0j\quad\hbox{on $\{|x|=1/2\}$}.
\ee
Moreover, $\frac{u_j}{d}(0)=e^j$ and $\displaystyle\lim_{|x|\to 1}\frac{u_j}{d}=|\tilde\varphi'(1)|>0$ by the Hopf-Oleinik lemma applied to $\varphi$,
hence
\be \label{optimBHI}
\frac{\sup_{B_1}(u_j/d)}{\inf_{B_1}(u_j/d)}\ge|\tilde\varphi'(1)|^{-1}e^j.
\ee

$\bullet$ For problem \eqref{optimLinfty1ci}, we choose $h=\varphi^2$, hence
$$\Delta u_j
=\varphi\bigl[4j|\nabla\varphi|^2+j\Delta h+j^2|\nabla h|^2-\lambda_0\bigr]e^{j h(x)},$$
so if we set
$$c_j:=\frac{-\Delta u_j}{u_j}=\lambda_0-4j|\nabla\varphi|^2-j\Delta h-j^2|\nabla h|^2,$$
we have $\Delta u_j+c_ju_j=0$.
There exist $C_1, C_2>0$ such that $C_1j^2\le\|c_j\|_\infty\le C_2j^2$ for all sufficiently large $j$.
This along with \eqref{optimLogg}-\eqref{optimBHI} and Proposition~\ref{lemr0ul}(i) yields \eqref{optimLinfty23}.
\smallskip

$\bullet$ For problem \eqref{optimLinfty1bi}, we take $h$ a radial decreasing $C^\infty$ function such that $h=1$ for $0\le |x|\le 1/3$
and $h=0$ for $2/3\le |x|\le 1$ and $|\nabla h|>C_0$ for $x=1/2$. We have
\be \label{Deltaujbj1}
\Delta u_j+\lambda_0 u_j=
\bigl[2j\nabla\varphi\cdot\nabla h+\varphi(j\Delta h+j^2|\nabla h|^2)\bigr]e^{j h(x)}.
\ee
Similarly to \eqref{gradsy}, since $\varphi$ is strictly radially decreasing and $h$ is radially decreasing, we have $|\nabla\varphi+j\varphi\nabla h|
\ge \max\{|\nabla\varphi|, j|\varphi\nabla h|\}>0$ in $\overline B_1\setminus\{0\}$.
We may then set
\be \label{Deltaujbj2}
b_j=-\frac{\nabla\varphi+j\varphi\nabla h}{|\nabla\varphi+j\varphi\nabla h|^2}
\bigl[2j\nabla\varphi\cdot\nabla h+\varphi(j\Delta h+j^2|\nabla h|^2)\bigr],\quad x\in\overline B_1\setminus\{0\}.
\ee
Since $h=1$, hence $b_j=0$, in $B_{1/3}$, the function $b_j$ extends to a function $b_j\in C^\infty(\overline B_1)$,
and \eqref{optimLogg-nabla}, \eqref{Deltaujbj1} and \eqref{Deltaujbj2} imply
$$\Delta u_j+b_j\cdot\nabla u_j+\lambda_0 u_j=0.$$
Moreover,
$$\begin{aligned}|b_j|
& =\frac{\bigl|2j\nabla\varphi\cdot\nabla h+\varphi(j\Delta h+j^2|\nabla h|^2)\bigr|}{|\nabla\varphi+j\varphi\nabla h|}
\le j\frac{\bigl|2\nabla\varphi\cdot\nabla h+\varphi\Delta h\bigr|}{|\nabla\varphi|}
+\frac{j^2\varphi|\nabla h|^2}{j|\varphi\nabla h|}\\
&\le j\frac{\bigl|2\nabla\varphi\cdot\nabla h+\varphi\Delta h\bigr|}{\min_{1/3\le|x|\le1}|\nabla\varphi|}+j|\nabla h|\le Cj.
\end{aligned}$$
We have $M_j\to\infty$ as in the previous proof. This along with \eqref{optimLogg}-\eqref{optimBHI} and Proposition~\ref{lemr0ul}(i)
yields \eqref{optimLinfty23}.
\end{proof}

We now turn to the partial optimality of the differential Harnack estimate \eqref{concllgeu2} in the inhomogeneous case $f\ne 0$.

First of all, as mentioned in Remark~\ref{remloggrad1d2}, \eqref{concllgeu2} is false without the assumption $f\ge 0$, even for $n=1$.
It suffices to consider for instance the functions $u,f\in C^\infty([0,1])$, with $u>0$ in $(0,1)$, given by
\be\label{counterfpositive}
u(x)=\exp(-x^{-1}),\qquad f(x):=-x^{-3}\bigl(2+x^{-1}\bigr)\exp(-x^{-1}),
\ee
which satisfy
$-u''=f$ in $(0,1)$, whereas $\frac{x|u'|}{u}=x^{-1}$ is unbounded.

Next, we have the following proposition which shows that, in the case $n=2$, for any $p>1$ estimate \eqref{concllgeu2} fails if $\|f\|_{L^1_d}$
is replaced by $\|f\|_{L^p_d}$ and $q$ is strictly larger but close to~$n$.
In other words, in this situation $C^1$-estimates are available but our logarithmic gradient estimate fails.
In dimension $n\ge 3$, we have the same conclusion only for $p>n/2$, so that the optimality
of the $L^1_d$ norm in \eqref{concllgeu2} remains unclear in this case.

\begin{prop} \label{PropLoggradOptimality}
Let $n\ge 2$ and $\frac{n}{2}<p<n<q<\frac{np}{n-p}$.
There exists a sequence of functions $f_k\in C^\infty(\overline B_1)$
with $f_k>0$ in $\overline B_1$, such that the solutions $u_k>0$ of
$$
-\Delta u_k=f_k \;\mbox{ in } B_1,\qquad
u_k=0  \;\mbox{ on } \partial B_1,
$$
satisfy, for all $k\ge 1$,
$$\sup_{B_1}\frac{d|\nabla u_k|}{u_k} \ge k\Bigl(1+\frac{\|f_k\|_{L^q(B_1)}}{\|f_k\|_{L^p_d(B_1)}}\Bigr).$$
\end{prop}

\begin{proof} [Proof of Proposition~\ref{PropLoggradOptimality}]
Let $\eps>0$. Define
$$\phi(r)=(r^2+\eps)^\gamma,\quad u(x)=u_\eps(x)=(1+\eps)^\gamma-\phi(r)$$
with $\gamma\in(0,1)$ to be chosen below. For $r>0$, we compute
\be\label{computphiprime}
\phi'=2\gamma r(r^2+\eps)^{\gamma-1},
\ee
$$\begin{aligned}
\phi''
&=2\gamma (r^2+\eps)^{\gamma-1}+4\gamma(\gamma-1)r^{2}(r^2+\eps)^{\gamma-2}\\
&=2\gamma (r^2+\eps)^{\gamma-2}\bigl[(r^2+\eps)+2(\gamma-1)r^2\bigr].
\end{aligned}$$
Therefore,
$$\begin{aligned}
f:=-\Delta u
&=\phi''+(n-1)r^{-1}\phi'
=2\gamma (r^2+\eps)^{\gamma-2}\bigl[n(r^2+\eps)+2(\gamma-1) r^2\bigr] \\
&=2\gamma (r^2+\eps)^{\gamma-2}\bigl[(n-2+2\gamma)r^2+n\eps],
\end{aligned}$$
hence
\be\label{computf}
2\gamma n (r^2+\eps)^{\gamma-1}\ge f\ge 4\gamma(r^2+\eps)^{\gamma-2}\bigl[\gamma r^2+\eps\bigr]\ge 4\gamma^2(r^2+\eps)^{\gamma-1}>0.
\ee
Let $p\ge 1$ and
$\eps\in(0,1/4)$.
Denoting by $C$ a generic positive constant depending only on $p$ and $\gamma$, we compute
$$
\|f\|^p_{L^p_d}\ge C \int_0^1 (r^2+\eps)^{(\gamma-1)p}r^{n-1}(1-r)dr
\ge C\eps^{(\gamma-1)p} \int_0^{\sqrt{\eps}} r^{n-1}dr
= C\eps^{(\gamma-1)p+n/2}
$$
hence
$\|f\|_{L^p_d}\ge C\eps^{\gamma-1+\frac{n}{2p}}$.
Also, if $q\ge 1$ and ${2(\gamma-1)q+n}<0$, then
$$\begin{aligned}
\|f\|_q^q
&\le C \int_0^1 (r^2+\eps)^{(\gamma-1)q}r^{n-1}dr\\
&{\le C\eps^{(\gamma-1)q} \int_0^{\sqrt{\eps}} r^{n-1}dr
+C \int_{\sqrt{\eps}}^1 r^{2(\gamma-1)q+n-1}dr\le C\eps^{(\gamma-1)q+n/2}},
\end{aligned}$$
hence $\|f\|_q\le C\eps^{\gamma-1+\frac{n}{2q}}$,
so that
$$\frac{\|f\|_q}{\|f\|_{L^p_d}}
\le C \eps^{-\frac{n}{2}(\frac{1}{p}-\frac{1}{q})}.$$
But on the other hand, by \eqref{computphiprime}, we have
$$\sup_{B_1}\frac{d|\nabla u|}{u}\ge \Bigl[\frac{d|\nabla u|}{u}\Bigr]_{|x|=\eps^{1/2}} \ge C|\nabla u|_{|x|=\eps^{1/2}}
\ge C\eps^{\gamma-\frac12}.$$
By our assumption, which implies $\frac{1}{p}-\frac{1}{q}<\frac{1}{n}$,
we may choose $\gamma>0$ small such that $0<\gamma<1-\frac{n}{2q}$ and
$\frac{n}{2}(\frac{1}{p}-\frac{1}{q})<\frac12-\gamma$, hence
$\eps^{-\frac{n}{2}(\frac{1}{p}-\frac{1}{q})}\ll \eps^{\gamma-\frac12}$ as $\eps\to 0$.
\end{proof}

\section{Appendix} \label{sec-app}

In this appendix we state and/or prove a number of auxiliary or technical results that we have used, and which
were postponed in order not to interrupt the main line of arguments.

\subsection{First eigenvalue}
Let $\Omega$ be a bounded
domain of $\rn$. Using the weak version of the Krein-Rutman theorem (see for instance \cite[Proposition 5.4.32]{DM}), it was proved in \cite{Ch1}, \cite{Ch2}, that any operator $\ld$ satisfying  \eqref{hyp1}, $b_1,b_2\in L^q(\Omega)$, $c,f\in L^{q/2}(\Omega)$, $q>n$,  possesses a first eigenvalue $\lambda_1=\lambda_1(-\ld,\Omega)$ (for the reader's convenience we note that what we call $\lambda_1$ is $-\lambda_1$ in \cite{Ch1}, \cite{Ch2}). This eigenvalue has the usual basic properties, specifically, it is proved in these works that  $\lambda_1$ is a simple eigenvalue, has smallest real part among all eigenvalues, decreases (strictly) with respect to the domain, and corresponds to a positive eigenfunction $\varphi_1\in H^1_0(\Omega)$. We have the characterization
\begin{equation}\label{charlambda1}
\lambda_1=\mathrm{sup}\{\lambda>0\::\: \mbox{there exists } w\in H^1(\Omega),\; w>0,\;  (\ld+\lambda)w\le0 \mbox{ in }\Omega\}.
\end{equation}
Furthermore the validity of the maximum principle for $\ld$ in $\Omega$ is equivalent to the existence of a nonnegative solution $w$ of
$\ld w<0$
 (or nontrivial nonnegative solution of
 $\ld w\le 0$)
 and hence  {\it the positivity of $\lambda_1$ is equivalent to the validity of the maximum principle for $\ld$ in~$\Omega$}.  The latter in turn easily implies that $\lambda_1>0$ guarantees the general solvability of the Dirichlet problem for \eqref{defdiv}  (see for instance the beginning of the proof of Proposition 4.1 in \cite{SS2}). We also know that $\varphi_1\in C^\alpha(\Omega)$, by \cite[Theorem 8.29]{GT} and the remark at the end of \cite[Section 8.10]{GT}. Under \eqref{hyp1}-\eqref{hyp2}, $\varphi_1\in C^{1,\alpha}(\Omega)$, by \cite[Section 8.11]{GT}}, \cite[Chapter 5.5]{Mo}.

The following proposition gives standard upper and lower  bounds on the first eigenvalue, which are used in the proof of
the Morel-Oswald estimate.

\begin{prop}\label{lowerbdeig}
Assume \eqref{hyp1}-\eqref{hyp2}, with $\Omega=B_\rho$ for $\rho>0$. We have
\begin{equation}\label{bdseig}
c_0|\Omega|^{-2/n} - C_0(\|b_1\|_{L^q(\Omega)}^{2\beta_q} + \|b_2\|_{L^q(\Omega)}^{2\beta_q} +\|c\|_{L^{q/2}(\Omega)}^{\beta_q})\le \lambda_1(-\ld, \Omega) \le \overline{C_0} ,
\end{equation}
 where $c_0, C_0>0$ depend on $n,\lambda,q$;  and $\overline{C_0}$ depends on $n,\lambda, \Lambda, q,\rho$, and upper bounds for $\|A\|_{C^\alpha(\Omega)}$, $\|b_1\|_{L^q(\Omega)} $, $\|b_2\|_{L^q(\Omega)} $, $\|c\|_{L^{q/2}(\Omega)}$.
\end{prop}

\noindent{\it Proof.} The lower bound follows from a well-known computation (see for instance \cite[Lemma 8.4]{GT} for the case $q=\infty$). Indeed, the bilinear form associated to $\ld$ satisfies, by a standard application of H\"older, Young and Sobolev-Poincaré inequalities (with $p=q/2>n/2$, $B=\|b_1^2+b_2^2+c\|_{L^p}$)
\begin{align*}
-(\ld u, u) &\ge c_0 \int_\Omega |\nabla u|^2 - C_0 \int_\Omega (b_1^2+b_2^2+c)u^2\ge c_0\|\nabla u\|_{L^2}^2 -C_0B\|u\|_{L^{2p^\prime}}^2\\
&\ge c_0\|\nabla u\|_{L^2}^2+c_0\|u\|_{L^{2^*}}^2 - C_0 B\left[(c_0/2C_0B)^{1/2}\|u\|_{L^{2^*}}+ C_0B^{\mu/2}\|u\|_{L^{2}}\right]^2 \\
&\ge c_0|\Omega|^{-2/n}\| u\|_{L^2}^2 - C_0B^{1+\mu}\| u\|_{L^2}^2=: \bar B\|u\|_{L^2}^2,\qquad \mu=\frac{1/2-1/(2p^\prime)}{1/(2p^\prime)-1/2^*},
\end{align*}
with constants $c_0, C_0$ which depend only on $n,\lambda, q$, and change from line to line.  Hence the operator $-\hat\ld=-\ld-\bar B+\epsilon$ is coercive for each $\epsilon>0$, thus satisfies the maximum principle, that is, $-\hat\ld$ has a positive first eigenvalue. Therefore $\lambda_1(-\ld, \Omega)\ge c_0|\Omega|^{-2/n} - c_0B^{1+\mu}-\epsilon$, which implies the first inequality in \eqref{bdseig}.

For the upper bound one can use a standard blow-up argument.
Assume for contradiction that  for each integer $j\ge 1$, there exist coefficients $A_j,b_{1,j},b_{2,j},c_j$ such that,
for the corresponding operator $\ld_j$, the eigenvalue $\mu_j:=\lambda_1(-\ld_j, B_\rho)\to\infty$ as $j\to\infty$.
Let $\varphi_j>0$ be the corresponding eigenfunction normalized by
\be\label{mujlarge2}
\sup_{B_\rho} \varphi_j= \varphi_j(x_j)=1
\ee
and rescale
$\psi_j(y) = \varphi_j(x)  = \varphi_j(x_j+r_jy)$, where $r_j = 1/\sqrt{\mu_j} \to 0$.
Then $\psi_j$  satisfies
\be\label{mujlarge3}
-\tilde\ld_j\psi_j=\psi_j\quad\hbox{in $G_j:=r_j^{-1}(B_\rho-x_j)$,}
\ee
where the coefficients of $\tilde\ld_j$ are given by
$$\tilde A_j(y)=A(x_j+r_jy),\quad \tilde b_{i,j}(y)=r_jb_{i,j}(x_j+r_jy),\quad \tilde c_j(y)=r_j^2c(x_j+r_jy).$$
Note that $G_j\to G$ where $G$ is either the whole space or a half-space.
For each fixed $L\ge 1$ and large $j$, we have
$$[\tilde A_j]_{\alpha, B_L\cap  G_j}=r_j^\alpha[A_j]_{\alpha, B_{Lr_j}(x_j) \cap B_\rho}
\le  r_j^\alpha[A_j]_{\alpha, B_\rho}
$$
$$\|\tilde c_j\|_{L^{q/2}(B_L\cap G_j)}=r_j^{2-\frac{n}{q}}\|c_j\|_{L^q( B_{Lr_j}(x_j)\cap B_\rho)}
\le r_j^{2-\frac{n}{q}}\|c_j\|_{L^q(  B_\rho)},
$$
(and similarly for $b_{1,j}$, $b_{2,j}$ in $L^q$),
hence $\|\tilde c_j\|_{L^{q/2}(B_L\cap G_j)}, \|\tilde b_{k,j}\|_{L^q(B_L\cap G_j)}, [\tilde A_j]_{\alpha, B_L\cap  G_j}\to 0$ as $j\to\infty$.
By \eqref{mujlarge2}, \eqref{mujlarge3}, Morrey's $W^{1,q}$-estimates (see \cite[Section~5.5]{Mo}) and compact embeddings, up to a subsequence $x_j\to x_0\in \bar B_1$,
we have $A_j\to A^0$ in $C_{loc}^{\alpha/2}(\overline G)$,
where $A^0=A(x_0)$ is a constant matrix satisfying \eqref{hyp1},
and $\psi_j\to \psi^0\ge 0$ weakly in $W^{1,q}( G)$ (and strongly in $C_{\mathrm{loc}}(G)$), where $\psi^0$ satisfies $\psi^0(0)=1$ and
$$
-\ld^0\psi^0= -\mathrm{tr}(A^0D^2\psi^0)= -\mathrm{div}(A^0 \nabla\psi^0) = \psi^0\quad \hbox{ in $G$.}
$$
By the standard characterization of the first eigenvalue \eqref{charlambda1} this implies that the first eigenvalue of
$-\ld^0$ is larger or equal to $1$ in any subdomain of $G$.
But $G$ contains balls of arbitrary radius $R$, which leads to a contradiction, since the self-adjoint operator $-\ld^0$ has constant coefficients, and hence
$$
\lambda_1(-\ld^0, B_R) = \frac{\lambda_1(-\ld^0, B_1)}{R^2}
\to 0 \; \mbox{ as } \; R\to \infty.\qquad\Box
$$

\subsection{Properties associated with uniformly local norms and scaling}

In this subsection we prove Propositions~\ref{basicr0}--\ref{lemr0ul}.

\begin{proof}[Proof of Proposition~\ref{basicr0}]
The functions
$$h_1(r)=[A]_{\alpha, r, \Omega}, \quad  h_2(r)= [b_1]_{\alpha, r, \Omega},
\quad h_3(r) =\|b_2\|_{q,r,\Omega}, \quad h_4(r) = \|c\|_{q, r, \Omega}$$
are clearly nondecreasing on $[0,r_\Omega]$.
We claim that they are continuous.

The continuity of $h_3, h_4$ on the left is easy to show by using monotone convergence, whereas the continuity on the right follows from the fact that $\|c\|_{L^q(B_r(x_r)\setminus B_{r_0}(x_r)\cap\Omega)}\to 0$ if $r\searrow r_0$, $x_r\in \Omega$.
The continuity of $h_1, h_2$ on the left follows easily from the continuity of $A, b_1$.
Assume for contradiction that $h_2$ is not continuous on the right (the argument for $h_1$ is the same). Then there exist $r_0\in[0,r_\Omega)$,
$\eta>0$ and sequences $r_i\to r_0$ and $x_i\in\overline \Omega$, $y_i\ne z_i\in\Omega$ with $|y_i-x_i|, |z_i-x_i|<r_i$,
such that $|y_i-z_i|^{-\alpha}|b_1(y_i)-b_1(z_i)|\ge h_2(r_0)+\eta$.
By passing to a subsequence we may assume that $x_i\to x_0$, $y_i\to y_0$, $z_i\to z_0$ for some $x_0,y_0,z_0\in \overline \Omega$.
Note also that $b_1$ extends to a $C^\alpha$ function on $\overline \Omega$ and that $B_r(x)\cap \Omega$ can be replaced by
$\overline B_r(x)\cap \overline \Omega$ in definition \eqref{defHbracket}.
If $y_0\ne z_0$, then $h_2(r_0)+\eta\le |y_0-z_0|^{-\alpha}|b_1(y_0)-b_1(z_0)|\le h_2(r_0)$ (where the last inequality follows
from $|y_0-x_0|, |z_0-x_0|\le r_0$, and definition \eqref{defHbracket} with $x=x_0$): a contradiction.
If $y_0=z_0$, then we have $|y_i-y_0|, |z_i-y_0|<r_0$ for $i$ large enough, hence
$|y_i-z_i|^{-\alpha}|b_1(y_i)-b_1(z_i)|\le h_2(r_0)$, which is again a contradiction.
The claim is proved.

Let now
$$h(r)=r\Bigl(r^{-1}_\Omega +
[A]^{1/\alpha}_{\alpha, r, \Omega} + \|b_1\|_{L^\infty(\Omega)} + [b_1]^{1/(\alpha+1)}_{\alpha, r, \Omega}+
\|b_2\|^{\beta_q}_{q,r,\Omega} + \|c\|^{\gamma_q}_{q, r, \Omega}\Bigr), \quad r\in [0,r_\Omega].$$
By the above, $h$ is strictly increasing and continuous on $[0,r_\Omega]$ and, moreover, $h(0)=0$ and $h(r_\Omega)\ge 1$.
Consequently there exists a unique $r\in (0,r_\Omega]$ such that $h(r)=1$ and $r_0=r$, which implies \eqref{relMr0}.
\end{proof}

\begin{proof}[Proof of Proposition~\ref{scaleinvtilde0}]
If $r_0(\ld,\omega)\ge \theta^{-1} r_0(\ld,\Omega)$,
then \eqref{relMr0} and the assumption $r_{\Omega}\ge \theta r_{\omega}$ yield
$$M(\ld,\omega)=r_0^{-1}(\ld,\omega)-r_{\omega}^{-1}\le  \theta r_0^{-1}(\ld,\Omega)-\theta r_{\Omega}^{-1}
= \theta M(\ld,\Omega)\le M(\ld,\Omega).$$
If $r_0(\ld,\omega)\le \theta^{-1} r_0(\ld,\Omega)$, since any ball of radius
$r_0(\ld,\omega)$ can be covered by $C(n) \theta^{-n}$ balls of radius $r_0(\ld,\Omega)$,
it follows from \eqref{deful}-\eqref{defM}
that $M(\ld,\omega)\le C(n,p,q,\alpha,\theta)M(\ld,\Omega)$.
\end{proof}

\begin{proof}[Proof of Proposition~\ref{scaleinvtilde}]
For any $\tilde r>0$,
we have
$$\begin{aligned}
[\tilde A]_{\alpha, \tilde r, B_1}
&=\sup_{x\in\overline{B_1}}\sup_{y,z\in B_{\tilde r}(x)\cap \Omega} |y-z|^{-\alpha}|A(Ry)-A(Rz)| \\
&=R^{\alpha}\sup_{\hat x\in\overline{B_R}}\sup_{\hat y,\hat z\in B_{R\tilde r}(\hat x)\cap \Omega} |\hat y-\hat z|^{-\alpha}|A(\hat y)-A(\hat z)|
= R^{\alpha}[A]_{\alpha, R\tilde r, B_R}.
\end{aligned}$$
This and a similar argument for $b_1$ yields
\begin{equation}\label{scaleinvtildePf0}
[\tilde A]_{\alpha, \tilde r, B_1}= R^{\alpha}[A]_{\alpha, R\tilde r, B_R},\
[\tilde b_1]_{\alpha, \tilde r, B_1}= R^{1+\alpha}[b_1]_{\alpha, R\tilde r, B_R},\
\|\tilde b_1\|_{L^\infty(B_1)}=R\|b_1\|_{L^\infty(B_R)}.
\end{equation}
On the other hand, for $x\in B_1$, we have
$$\|\tilde b_2\|_{L^q(B_{\tilde r}(x)\cap B_1)}=R^{1-n/q}\|b_2\|_{L^q(B_{R\tilde r}(Rx)\cap B_R)},\quad
\|\tilde c\|_{L^q(B_{\tilde r}(x)\cap B_1)}=R^{2-n/q}\|c\|_{L^q(B_{R\tilde r}(Rx)\cap B_R)}$$
hence
\begin{equation}\label{scaleinvtildePf1}
\|\tilde b_2\|_{q,\tilde r,B_1}=R^{1-n/q}\|b_2\|_{q,R\tilde r,B_R},\quad
\|\tilde c\|_{q,\tilde r,B_1}=R^{2-n/q}\|c\|_{q,R\tilde r,B_R)}.
\end{equation}
Using also $r_{B_R}=c(n)R$, it follows that, for any $\tilde r>0$,
$$\tilde r\bigl(r^{-1}_{B_1} + [\tilde A]^{\frac{1}{\alpha}}_{\alpha, \tilde r, B_1} + \|\tilde b_1\|_{L^\infty(B_1)} + [\tilde b_1]^{\frac{1}{\alpha+1}}_{\alpha, \tilde r, B_1}
+\|\tilde b_2\|^{\beta_q}_{q,\tilde r,B_1} + \|\tilde c\|^{\gamma_q}_{q, \tilde r,B_1}\bigr)\le 1 \Longleftrightarrow $$
$$\hbox{$r:=R\tilde r$ satisfies }
r \bigl(r^{-1}_{B_R}+ [A]^{\frac{1}{\alpha}}_{\alpha, r, B_R} + \|b_1\|_{L^\infty(B_R)} + [b_1]^{\frac{1}{\alpha+1}}_{\alpha, r, B_R}
+\|b_2\|^{\beta_q}_{q,r,B_R} + \|c\|^{\gamma_q}_{q,r,B_R}\bigr)\le 1.$$
This yields \eqref{scaleinvtilde2}.
Combining \eqref{scaleinvtilde2} and \eqref{scaleinvtildePf0} gives \eqref{scaleinvtilde2b}.
Next, by \eqref{scaleinvtilde2}, \eqref{scaleinvtildePf1},
$$\|\tilde b_2\|_{q,\tilde r_0,B_1}=R^{1-n/q}\|b_2\|_{q,r_0,B_R}$$
which, together with the similar properties for $\tilde c$, $\tilde f$, gives \eqref{scaleinvtilde4}.
\end{proof}

 \begin{proof}[Proof of Proposition~\ref{lemr0ul}]
(i) Recalling the definition \eqref{defM} and Proposition \ref{basicr0},
we have
$$ r_0^{1-n/q} \|b_2\|_{q,r_0,B_R}\le 1,\quad
r_0^{2-n/q} \|c\|_{q,r_0,B_R}\le 1.$$
By H\"older's inequality it follows that
$$ \|b_2\|_{q,r_0,B_R} \le C(n)r_0^{n/q} \|b_2\|_{L^\infty(B_R)}
\le C(n)\|b_2\|^{-\frac{n}{q-n}}_{q,r_0,B_R}\|b_2\|_{L^\infty(B_R)},$$
hence
$$ \|b_2\|^{\frac{q}{q-n}}_{q,r_0,B_R}
\le C(n)\|b_2\|_{L^\infty(B_R)}$$
i.e. \eqref{culinfty} (the statement for $c$  is obtained similarly).

\smallskip

(ii) The upper estimates follow from assertion (i) (since the $L^\infty$ norm does not involve $r_0$).
To prove the lower estimate for $b$ (the case of $c$ is similar)
set $K_\lambda=\|\lambda b\|_{q,r_\lambda,B_R}\equiv \sup_{x\in B_R} \|\lambda b\|_{L^q(B_R\cap B_{{r_\lambda}}(x))}$.
By \eqref{relMr0}, we have
\be\label{Krlambda0}
r_\lambda\Bigl(r_{B_R}^{-1}+\lambda^{\beta_q}\sup_{x\in B_R} \|b\|^{\beta_q}_{L^q(B_R\cap B_{{r_\lambda}}(x))}\Bigr)
=r_\lambda(r_{B_R}^{-1}+K_\lambda^{\beta_q})=1=r_1(r_{B_R}^{-1}+K_1^{\beta_q}).
\ee
Assume $\lambda\ge 1$. Then \eqref{Krlambda0} implies
$$r_\lambda\Bigl(r_{B_R}^{-1}+\sup_{x\in B_R} \|b\|^{\beta_q}_{L^q(B_R\cap B_{{r_\lambda}}(x))}\Bigr)
\le r_1\Bigl(r_{B_R}^{-1}+\sup_{x\in B_R} \|b\|^{\beta_q}_{L^q(B_R\cap B_{{r_1}}(x))}\Bigr).
$$
By the definition of $r_1$ we get $r_\lambda\le r_1$ and, going back to \eqref{Krlambda0}, that $r_\lambda K_\lambda^{\beta_q} \ge r_1K_1^{\beta_q}$, hence
\be\label{Krlambda}
r_1^{-n/q}K_1^{-n\beta_q/q}K_\lambda^{n\beta_q/q}\ge r_\lambda^{-n/q}.
\ee
On the other hand, since any ball with radius $r_1$ can be covered by $C(n)(r_1/r_\lambda)^n$ balls with radius $r_\lambda$, we have
$$K_1=\sup_{x\in B_R} \|b\|_{L^q(B_R\cap B_{{r_1}}(x))}
\le C(n)  \Bigl(\frac{r_1}{r_\lambda}\Bigr)^{n/q}\sup_{x\in B_R} \|b\|_{L^q(B_R\cap B_{{r_\lambda}}(x))}
=C(n) \Bigl(\frac{r_1}{r_\lambda}\Bigr)^{n/q}\lambda^{-1} K_\lambda.$$
Combining this with \eqref{Krlambda}, we get
$$r_1^{-n/q}K_1^{-n\beta_q/q}K_\lambda^{1+n\beta_q/q} \ge r_\lambda^{-n/q}K_\lambda\ge C(n) \lambda r_1^{-n/q}K_1,$$
hence $K_\lambda^{1+n\beta_q/q} \ge C(n) \lambda K_1^{1+n\beta_q/q}$.
Since $1+\frac{n\beta_q}{q}=(1-\frac{n}{q})^{-1}$ this yields the lower estimate for $b$.
\smallskip

Let us finally check the continuity statement.
 As above, by \eqref{Krlambda0} $r_\lambda$ is nonincreasing with respect to $\lambda\in(0,\infty)$
and \eqref{Krlambda0} guarantees that
$$r_\lambda H(\lambda)=\lambda^{-\beta_q},
\quad\hbox{ where }
H(\lambda):=\lambda^{-\beta_q}{r_{B_R}^{-1}}+ \sup_{x\in B_R} \|b\|^{\beta_q}_{L^q(B_R\cap B_{{r_\lambda}}(x))}.$$
Since the functions $r_\lambda$ and $H(\lambda)$ are nonincreasing on $(0,\infty)$
and their product is a continuous function,
both functions are necessarily continuous. This implies that $K_\lambda=\|\lambda b\|_{q,r_\lambda,B_R}
=\lambda \sup_{x\in B_R} \|b\|_{L^q(B_R\cap B_{{r_\lambda}}(x))}$ is itself continuous.
\end{proof}

\subsection{Harnack inequalities}

 In this section we provide a proof of the global Harnack inequalities in Theorem \ref{BHIoptim}.
 They rely on modifications of Harnack chain arguments from \cite{SS2},
adapted to general domains (and not only for balls as in \cite{SS2}).
To this end, we need the following proposition,
which guarantees the existence of suitable coverings and associated Harnack chains,
with optimal length estimate in terms of the geodesic diameter.
 This may be known but we could not find a reference in the literature,
so we provide a proof.
Here we use the notation given at the beginning of Section \ref{sec-main}.

\begin{prop}\label{geodes}
 (i) Let $\Omega$ be an arbitrary bounded domain of $\R^n$ and let $r\in(0,D)$.
There exist constants $c_0,c_1>0$ depending only on $n$, integers $N,I\ge 2$ satisfying
\bel{coverA0}
N=N_r\le 2+\frac{6D}{r},\qquad I=I_r\le c_0\Bigl(1+\frac{D_0}{r}\Bigr)^n\,,
\ee
 points $x_1,\dots,x_I\in \Omega$ for which
\bel{coverA1}
\Omega\subset\bigcup_{j=1}^I B_r(x_j),\quad\mbox{and}
\ee
\bel{coverA3}
\begin{aligned}
&\hbox{for each $k,l\in \{1,\ldots,I\}$ there exist $p\in\{2,\dots,N\}$ and $(j_1,\dots,j_p)\in \{1,\dots,I\}^p$}\\
&\hbox{such that $j_1=k$, $j_p=l$ and $|B_r(x_{j_{i-1}})\cap B_r(x_{j_i})|\ge c_1r^n$ for all $i\in \{2,\ldots,p\}$.}
\end{aligned}
\ee
\smallskip

(ii) Let $\Omega$ be a bounded domain of $\R^n$ with $C^{1,\bar\alpha}$ boundary and let $r\in(0,r_\Omega)$.
There exist constants $c_0,c_1>0$ depending only on $n$, integers $N,I\ge 2$
satisfying \eqref{coverA0}
and points $x_1,\dots,x_I$ satisfying \eqref{coverA1}, \eqref{coverA3}, such that
\bel{coverA2}
\hbox{for each $j\in\{1,\dots,I\}$, either dist$(x_j, \partial \Omega)\ge 3r/2$ or $x_j\in \partial \Omega$, and}
\ee
\bel{coverA2b}
\hbox{for each $j\in\{1,\dots,I\}$, $|\Omega\cap B_r(x_j)|\ge c_0r^n$}.
\ee
\end{prop}

\begin{rem}
The upper bound \eqref{coverA0} on $N$ in Proposition~\ref{geodes}
(which implies $N\le \frac{8D}{r}$) is qualitatively optimal,
since any $N$ with such properties necessarily satisfies
\bel{optimN}
N\ge \frac{D}{2r}.
\ee
Indeed, fix any $x,y\in \Omega$, take $k,l$ such that $x\in B_r(x_k)$, $y\in B_r(x_l)$
and a chain $k=j_1<j_2<\dots<j_p=x_l$ with $p\le N$.
Then, using that $B_r(x_{j_{i+1}})\cap B_r(x_{j_i})\ne\emptyset$, we see that the geodesic distance between $x,y$ can be estimated by
$$d_\Omega(x,y)\le |x-x_k|+|y-x_l|+\sum_{i=2}^p |x_{j_{i}}-x_{j_{i-1}}|\le r+r+2(p-1)r=2pr\le 2Nr.$$
Taking supremum over $x,y\in \Omega$, we obtain \eqref{optimN}.
As for the upper bound on $I$ in \eqref{coverA0}, it is clearly optimal in general
(unless $\Omega$ has a specific, e.g.~tubular, geometry).
\end{rem}

\begin{proof}[Proof of Proposition~\ref{geodes}]

 We will only prove assertion~(ii), the proof of assertion~(i) being similar and easier.
We first claim that there exist a number $I\in\mathbb{N}$ and points $x_1,\dots,x_I$ such that
\be\label{coverA00}
 I\le c_0(n)\Bigl(1+\frac{D_0}{r}\Bigr)^n, \quad
\hbox{(\ref{coverA2}) holds,}\quad\hbox{ and }
\Omega\subset\bigcup_{j=1}^I B_{5r/6}(x_j)
\ee
(hence in particular \eqref{coverA1}).

For $\eps>0$, we denote
$$\Omega_\eps=\{x\in\Omega;\ {\rm dist}(x,\partial\Omega)>\eps\},\quad
\omega_\eps=\{x\in\Omega;\ {\rm dist}(x,\partial\Omega)<\eps\}.$$
\indent $\bullet$ Since $\{B_{4r/5}(x);\ x\in\partial\Omega\}$ is an open covering of the compact $\overline\omega_{3r/4}$,
we can cover $\overline\omega_{3r/4}$ by finitely many such balls, whose set of centers we denote by $\Sigma_1$.

$\bullet$ Also, we can obviously cover the compact $\overline\Omega_{3r/2}$ by finitely many balls
of radius $4r/5$ and centered at points $x$ with $d(x)\ge3r/2$. We denote the set of their centers by $\Sigma_2$.

$\bullet$ Next, consider the case when $x\in K:=\overline\omega_{3r/2}\setminus\omega_{3r/4}=\{3r/4\le d(x)\le 3r/2\}$.
Set $d=d(x)$, denote by $p_x$ the projection of $x$ on $\partial\Omega$ and by $\nu_x$ the outer normal vector at $p_x$,
hence
$x=p_x-d\nu_x$.
Let $\ell=\frac{10}{13}$ and $z=p_x-(d+\ell r)\nu_x$.
Note that $|x-z|=\ell r<4r/5$.
We claim that
\be\label{claimdz}
d(z)\ge 3r/2.
\ee
 Indeed, recalling the first paragraphs of Section \ref{sec-main} and the notation therein,
after an orthonormal change of coordinates,
we may assume that $p_x=0$, $x=(0,d)\in\R^{n-1}\times\R$, $z=(0,d+\ell r)$
and that
\be\label{claimdz2}
\Omega\supset\bigl\{y=(y';y_n)\in \R^{n-1}\times\R:\ |y'|< \bar\rho_\Omega\ \hbox{and}\ k_\Omega |y'|^{1+\alpha}<y_n< \bar\rho_\Omega\bigr\}.
\ee
Working in the new coordinate, for any $y$ such that $|y-z|<3r/2$,
using $r<r_\Omega\le\min(\bar\rho_\Omega/4,$ $(120 k_\Omega)^{-1/\alpha})$, we obtain $y_n<d+\ell r+3r/2<4r<\bar\rho_\Omega$, $|y'|<3r/2<\bar\rho_\Omega$
and
$$y_n>d+\ell r-3r/2\ge \ell r-3r/4=r/52\ge k_\Omega (3r/2)^{1+\alpha}\ge k_\Omega |y'|^{1+\alpha},$$
hence $y\in\Omega$ owing to \eqref{claimdz2}. This proves \eqref{claimdz}.
Consequently,
 $\{B_{4r/5}(z)\::\: d(z)\ge 3r/2\}$ is an open covering of the compact
$K$ and we can extract a finite covering, whose set of centers we denote by $\Sigma_3$.

Enumerating $\Sigma:=\Sigma_1\cup\Sigma_2\cup\Sigma_3$ as $\{x_1,\dots,x_m\}$, we then have
$$\Omega\subset\bigcup_{i=1}^m B_{4r/5}(x_i).$$
Since $m$ need not be bounded by $c_0(n)\bigl(\frac{D_0}{r}\bigr)^n$, we will now define a subset of $\{1,\dots,m\}$ as follows.
We may fix $\eps=\eps(n)>0$ such that, for any $y,z\in\R^n$, $|y-z|\le \eps r$ implies $B_{4r/5}(y)\subset B_{5r/6}(z)$.
Set $i_1=1$.
We first remove all $j$ with $1<j\le m$ such that $|x_j-x_{i_1}|< \eps r$,
and we note that $B_{4r/5}(x_j)\subset B_{5r/6}(x_{i_1})$ for all such $j$.
Let $i_2$ be the smallest remaining index $>i_1$ (if any).
We then remove all $j$ with $i_2<j\le m$ such that $|x_j-x_{i_2}|< \eps r$,
and we note that $B_{4r/5}(x_j)\subset B_{5r/6}(x_{i_2})$ for all such~$j$.
Repeating the process, we obtain $I\le m$ and $1=i_1<\dots<i_I\le m$ such that
\bel{disteps}
|x_{i_j}-x_{i_k}|\ge \eps r\quad\hbox{ for all $1\le j< k\le I$}
\ee
 and we have
\bel{cover1}
\Omega\subset\bigcup_{j=1}^I B_{5r/6}(x_{i_j}).
\ee
Picking $x_0$ such that $\Omega\subset B(x_0,D_0)$, we have
 $B_{\eps r/2}(x_{i_j})\subset B(x_0,D_0+\eps r/2)$.
Since \eqref{disteps} guarantees that $B_{\eps r/2}(x_{i_j})\cap B_{\eps r/2}(x_{i_k})=\emptyset$ for all $1\le j<k\le I$,
it follows that $I |B_{\eps r/2}(0)|\le |B_{D_0+\eps r/2}(0)|$, hence
$$I\le \Bigl(1+\frac{2D_0}{\eps r}\Bigr)^n\le c_1(n)\Bigl(1+\frac{D_0}{r}\Bigr)^n.$$
Relabelling these points, we have thus proved claim \eqref{coverA00}.

\smallskip

Let us next prove  \eqref{coverA3} for some $N\le 1+\frac{6D}{r}$.
We observe that, for any $x,y\in\R^n$,
\bel{measure-c0}
\bar B_r(x)\cap \bar B_{5r/6}(y)\ne\emptyset\Longrightarrow |B_r(x)\cap B_r(y)|\ge c_0(n)r^n.
\ee
Fix $1\le k<l\le I$.
By the definition of $D$, there exists a Lipschitz curve $\gamma:[0,1]\to \Omega$
such that $\gamma(0)=x_k$, $\gamma(1)=x_l$ and $s(1)\le D$,
where $[0,1]\ni t\mapsto s(t)$ denotes the increasing curvilinear abscissa along $\gamma$.
We set $t_1=0$, $j_1=k$.
\smallskip

$\bullet$
If $\gamma([t_1,1])\subset B_r(x_{j_1})$ then we set $j_2=l$ and $p=2$.
Since $x_l=x_{j_2}\in B_r(x_{j_1})$, \eqref{measure-c0} guarantees that
\bel{measure-c010}
|B_r(x_{j_1})\cap B_r(x_{j_2})|\ge c_0(n)r^n.
\ee

$\bullet$ Otherwise, there exists a minimal $t_2\in(t_1,1]$ such that $\gamma(t_2)\in \partial B_r(x_{j_1})$.
In particular, $s(t_2)-s(t_1)\ge|\gamma(t_2)-x_{j_1}|=r$.
And owing to \eqref{cover1}, there exists $j_2\in\{1,\dots,I\}$ such that $\gamma(t_2)\in B_{5r/6}(x_{j_2})$.
Moreover, since $\gamma(t_2)\in \bar B_r(x_{j_1})\cap \bar B_{5r/6}(x_{j_2})$, \eqref{measure-c0} guarantees that
\eqref{measure-c010} is still true.

\smallskip
$\bullet$ If $\gamma([t_2,1])\subset B_r(x_{j_2})$  then we set $j_3=l$, $p=3$
and, similar to the case $p=2$, we obtain
$|B_r(x_{j_2})\cap B_r(x_{j_3})|\ge c_0(n)r^n$.

\smallskip
$\bullet$ Otherwise, there exists a minimal $t_3\in(t_2,1]$ such that $\gamma(t_3)\in \partial B_r(x_{j_2})$.
In particular,
\bel{curvilineardiff}
s(t_3)-s(t_2)\ge |\gamma(t_3)-\gamma(t_2)| \ge |\gamma(t_3)-x_{j_2}|-|\gamma(t_2)-x_{j_2}|\ge r-(5r/6)=r/6.
\ee
Then, owing to \eqref{cover1}, there exists $j_3\in\{1,\dots,I\}$ such that $\gamma(t_3)\in B_{5r/6}(x_{j_3})$.

\smallskip
$\bullet$ We can repeat this process as long as $\gamma([t_i,1])\not\subset B_r(x_{j_i})$.
Also, at the $i$-th step,
since $\gamma(t_i)\in \bar B_r(x_{j_{i-1}})\cap \bar B_{5r/6}(x_{j_i})$, \eqref{measure-c0} guarantees that
\bel{measure-c02}
|B_r(x_{j_{i-1}})\cap B_r(x_{j_i})|\ge c_0(n)r^n
\ee
and, similar to \eqref{curvilineardiff}, we have $s(t_i)-s(t_{i-1})\ge r/6$.
Since $s$ is an increasing function with $s(1)-s(0)\le D$, it follows that $i$ cannot exceed the value $1+(6D/r)$
and we eventually reach $i$ such that $\gamma([t_i,1])\subset B_r(x_{j_i})$.
Consequently we obtain an integer $p\le 2+(6D/r)$ and indices
$k=j_0,j_1,\dots,j_p=l$ such that \eqref{measure-c02} holds for all $i\in\{1,\dots,l\}$.

Finally, property \eqref{coverA2b} is trivial in case dist$(x_i, \partial \Omega)\ge 3r/2$.
In case $x_i\in \partial \Omega$, it easily follows from
\eqref{claimdz2} and the definition of $k_\Omega, \bar\rho_\Omega$ at the beginning of Section~\ref{sec-main}.
This completes the proof.
\end{proof}

We turn to the proof of Theorem \ref{BHIoptim}, for which we will use the following particular case of the results in~\cite{GSS}. Here we denote $B_R^+ = B_R\cap \{x\::\: x_n>0\}$.

\begin{thm}[\cite{GSS}] \label{fromGSS} Assume $u\ge0$ in $B_1^+$ is a weak solution of $-\mathcal{L}u\ge f $ in $B_1^+$, for some $f\in L^q(B_1^+)$, and the coefficients of $\mathcal{L}$ satisfy \eqref{hyp1}-\eqref{hyp2} in $\Omega = B_1^+$.
There exist  constants $\epsilon, C>0$ depending only on $n,q,\alpha,\lambda,\Lambda$, and upper bounds on the $C^\alpha$-norms of $A,b_1$ and the $L^q$-norms of $b_2,c$ in $B_1^+$, such that for any $\omega\subset B_{3/4}^+$
\begin{equation}\label{fixWBHI}
\left(\int_{\omega} \left(\frac{u}{d}\right)^\epsilon\right)^{1/\epsilon} \le
C \left( \inf_{\omega} \frac{u}{d} +
\|f\|_{L^{q}(B_1^+)}\right).
\end{equation}
If on the other hand  $-\mathcal{L}u\le f $  in $B_1^+$, $u\le0$ on $\partial B_1^+$, then
\begin{equation}\label{fixlocmax}
\sup_{\omega} \frac{u^+}{d}\le C\left( \left(\int_{\omega} \left(\frac{u^+}{d}\right)^\epsilon\right)^{1/\epsilon} +
\|f\|_{L^{q}(B_1^+)} \right).
\end{equation}
\end{thm}

The inequality \eqref{fixWBHI} is a particular case of  \cite[Theorem 1.1]{GSS}, while \eqref{fixlocmax} is a simple application of the local maximum principle, see for instance \cite[p.9]{GSS}.

\begin{proof}[Proof of Theorem \ref{BHIoptim}]
We
consider the collection of balls with radius $r_0/2$ which covers $\overline{\Omega}$ ($r_0$ is the number from \eqref{defr0}, \eqref{relMr0}) and the numbers $I,N$
 given by Proposition~\ref{geodes} (ii) applied with $r=r_0/2$.

\noindent {\it Case 1.} Fix one such ball $B=B_{r_0/2}(\bar x)$ whose center $\bar x$ is on $\partial\Omega$. Let $B^\prime=B_{r_0}(\bar x)$ and $\Phi$ be the $C^{1,\bar\alpha}$ diffeomorphism defined in Section~\ref{sec-main} which sends $(B^\prime\cap\Omega-\bar x)/r_0$ to $B_1^+$.

For $x\in B^\prime\cap\Omega$, let $y=\Phi(\frac{x-\bar x}{r_0})\in B_1^+$. Obviously for any $\hat x, \tilde x\in B^\prime\cap\Omega$ we have
\begin{equation}\label{compdist0}(1/2r_0)|\hat x- \tilde x|\le \frac{1}{\sup|D\Phi^{-1}|r_0}|\hat x- \tilde x|\le |\hat y -\tilde y|\le \frac{\sup|D\Phi|}{r_0}|\hat x- \tilde x|\le (2/r_0)|\hat x- \tilde x|,
\end{equation}
and in particular for each $y\in B_1^+$
\begin{equation}\label{compdist}
(1/2r_0)\mathrm{dist}(x,\partial\Omega)\le\mathrm{dist}(y, B_1^0) \le (2/r_0)\mathrm{dist}(x,\partial\Omega).
\end{equation}

We now make the change of variable $y= \Phi(\frac{x-\bar x}{r_0})$, $x= \bar x+ r_0\Phi^{-1}(y)$ in the inequality $-\mathcal{L}u\ge f $ in $B^\prime\cap\Omega$. Setting $\hat u (y) = u(x)$, a straightforward computation shows that $\hat u$ satisfies an equation $-\widehat{\mathcal{L}}\hat u\ge \hat f $ in $B_1^+$, where the coefficient $\hat A$ satisfies \eqref{hyp1} with modified constants $0<\hat \lambda\le \hat \Lambda$ depending only on upper bounds for $|D\Phi|$, $|D\Phi^{-1}|$ (which in our case are bounded by 2), and coefficients $\hat b_1, \hat b_2, \hat c$ which satisfy \eqref{hyp2}. In addition,  the $C^\alpha$-norms of $\hat A,\hat b_1$ as well as the $L^q$-norms of $\hat b_2,\hat c$ in $B_1^+$ are bounded above by a universal constant which again depends only on $|D\Phi|$, $|D\Phi^{-1}|$. The latter fact is due to the choice of $r_0$ -- see  Proposition \ref{scaleinvtilde} and its proof where the particular case $\Phi=I$ is considered in detail. For instance, we have for some universal $C_0$
$$\begin{aligned}
\|\hat b_2\|_{L^q(B_1^+)}
&\le C_0 r_0\Bigl(\int_{B_1^+} |b_2(\bar x+r_0\Phi^{-1}(y))|^q\,dy\Bigr)^{1/q}\\
&\le C_0 r_0^{1-n/q}\Bigl(\int_{B^\prime\cap\Omega} |b_2(x)|^q |\mathrm{det}\,D\Phi|\,dx\Bigr)^{1/q}
\le C_0 r_0^{1-n/q}\|b_2\|_{L^q(B'\cap\Omega)}\le C_0
\end{aligned}
$$
(the last inequality follows from \eqref{defr0}). Note also that the change of variables sends $B$ to a subset $\omega$ of $B_{3/4}^+$, because of \eqref{compdist0}. Hence we can apply Theorem \ref{fromGSS} to $\hat u$ and $-\widehat{\mathcal{L}}\hat u\ge \hat f $ in $B_1^+$, getting \eqref{fixWBHI} for $\hat u(y) $ and $\hat d(y) = \mathrm{dist}(y, B_1^0)$. Changing back into the $x$ variable and using \eqref{compdist} we get
\begin{equation}\label{ineqbdry11}
\left(\int_{B\cap \Omega}\left(\frac{u}{d}\right)^\epsilon\right)^{1/\epsilon} \le
C_0 r_0^{n/\epsilon} \left( \inf_{B\cap \Omega} \frac{u}{d} +  r_0^{1-n/q}\|f\|_{L^{q}(B'\cap\Omega)}\right).
\end{equation}

\noindent {\it Case 2.} In the simpler case when a given ball $B=B_{r_0/2}(\bar x)$ is such that $\bar x\in \Omega$ (with dist$(\bar x,\partial\Omega)\ge 3r_0/4$) we use the change  $y= \frac{x-\bar x}{\tilde r_0}$,  $\tilde r_0 = \min\{r_0, \mathrm{dist}(\bar x,\partial\Omega)\}\in [3r_0/4, r_0]$, which sends $B_{\tilde r_0}$ to $B_1$ (and $B$ to $B_{\delta}$, $\delta\in[1/2,2/3]$), and the new operator again has uniformly bounded coefficients in the corresponding norms. By the interior weak Harnack inequality (as in Step 1 in the proof of  \cite[Theorem 2.1]{SS2}) this gives
\begin{equation}\label{ineqinside}
\left(\int_{B} u^\epsilon\,dx\right)^{1/\epsilon} \le C_0 {\tilde r_0}^{n/\epsilon} \left( \inf_{B} u + {\tilde r_0}^{2-n/q}\|f\|_{L^{q}(B')}\right).
\end{equation}
which implies
\begin{equation}\label{ineqbdry12}
\left(\int_{B}\left(\frac{u}{d}\right)^\epsilon\right)^{1/\epsilon} \le
C_0 r_0^{n/\epsilon} \left( \inf_{B} \frac{u}{d} +  r_0^{1-n/q}\|f\|_{L^{q}(B'\cap\Omega)}\right),
\end{equation}
since for $x\in B$ we have  dist$( x,\partial\Omega)\ge r_0/4$, hence $\sup_B d/\inf_B d\le (\inf_B d + r_0)/\inf_B d\le 5$.

Thanks to \eqref{ineqbdry11} and \eqref{ineqbdry12} we can repeat the iteration argument in \cite[p.12]{SS2}, and deduce that for any two balls $B_k, B_l$ from the covering, $k,l\in \{1,\ldots,I\}$,
\begin{equation}\label{iterineq}
\left(\int_{B_l\cap\Omega} u^\epsilon\,dx\right)^{1/\epsilon} \le C_0^N  \left( \left(\int_{B_k\cap\Omega} u^\epsilon\,dx\right)^{1/\epsilon} + { r_0}^{1-n/q+n/\epsilon}\|f\|_{L^{q}_{ul}(\Omega)}\right)
\end{equation}
(we note that $m$ and $d$ in \cite{SS2} are our $I$ and $N$ here). Summing over $l\in \{1,\ldots,I\}$, and then using  \eqref{ineqbdry11}-\eqref{ineqbdry12} again for each $k\in \{1,\ldots,I\}$,
along with \eqref{coverA0} and \eqref{relMr0}, we obtain~\eqref{sharpWBHI}.

The  full Harnack inequality \eqref{sharpBHI} follows exactly as in \cite[Step 3]{SS2} after the same rescaling as above and using \eqref{fixlocmax}, combined with \eqref{sharpWBHI}.
\end{proof}

\begin{rem} \label{remnoloss}
We note that the above argument actually yields $r_0^{1-n/q}$ and $\|f\|_{q,r_0,\Omega}$
on the right hand side of \eqref{sharpBHI},
instead of $D^{1-n/q}$ and $\|f\|_{L^q(\Omega)}$, respectively.
However, in view of the exponential factor in \eqref{sharpBHI}, of the inequality
$$\|f\|_{r_0,q,\Omega} \le \|f\|_{L^q(\Omega)}\le  c(n)(D/r_0)^{n/q} \|f\|_{r_0,q,\Omega}$$
(since one can cover $\Omega$ by $c(n)(D/r_0)^n$ balls of radius $r_0$),
and of the fact that
$(M+r^{-1}_\Omega)D=D/r_0$ by \eqref{relMr0},
there is no loss coming from this replacement.
\end{rem}

{\bf Acknowledgements.}
This work was partially done during visits of the second
 author to the Mathematics Departement of the
Pontificia Universidade Cat\'olica do Rio de Janeiro.
He thanks PUC-Rio for the hospitality
and also gratefully acknowledges financial support from the R\'eseau Franco-Br\'esilien de Math\'ematiques (RFBM). The first author was supported by grants CNPq 307772/2022-5 and FAPERJ E-26/204.317/2024.

\end{document}